\documentclass[12pt,a4paper,oneside]{article}
\usepackage[T1]{fontenc}
\usepackage[english]{babel}
\usepackage{amsmath}
\usepackage{amsfonts}
 \usepackage{amssymb}
 \usepackage{amsthm}
\usepackage{subfigure} %%added for multiple figures 
\usepackage[right=2.cm,left=2.cm,top=2cm,bottom=2cm,headsep=0.5cm,footskip=0.5cm]{geometry}
\usepackage{graphicx} % para gráficas
\usepackage{float} %para poner mi gráfica aquí
\usepackage{hyperref}

\theoremstyle{definition}
\newtheorem{theorem}{Theorem} 

\theoremstyle{definition}
\newtheorem{lemma}{Lemma}

\theoremstyle{definition}

\theoremstyle{definition}
\newtheorem{example}{Example}  

\theoremstyle{definition}

\title{Bifurcations of a Leslie Gower predator prey model with Holling type III functional response and Michaelis-Menten prey harvesting }

\author{Eric \'Avila-Vales\\ \'Angel Estrella-Gonz\'alez and Erika Rivero-Esquivel\\
\begin{small}
Facultad de Matem\'aticas, Universidad Aut\'onoma de Yucat\'an.
\end{small}  \\
\begin{small}
Anillo Perif\'erico Norte, Tablaje 13615, C.P. 97119.
M\'erida, M\'exico.
\end{small}}

\date{\today}

\begin{document}
 \maketitle

\begin{abstract}
We discuss the stability and bifurcation analysis for a predator-prey system with  non-linear Michaelis-Menten prey harvesting. The existence and stability of possible equilibria are investigated. We provide rigorous mathematical proofs for the existence of Hopf and saddle node bifurcations. We prove that the system exhibits Bogdanov-Takens bifurcation of codimension two, calculating the normal form. We provide several numerical simulations to illustrate our theoretical findings.
\end{abstract} 
 
\section{Introduction}

From the point of view of human needs, the exploitation of biological resources
and harvesting of populations are commonly practiced in fishery,
forestry, and wildlife management. Simultaneously, there is a wide range of
interests in the use of bioeconomic models to gain insight into the scientific
management of renewable resources which is related to the optimal management
of renewable resources. It is obvious that a harvesting in preys affects
the population of predators indirectly, because it reduces the food population
available in the area. There are basically three types of harvesting
reported in the literature \cite{gupta2015dynamical}.
\begin{itemize}
\item Constant harvesting, $h(x) = h$, where a constant number of individuals
are harvested per unit of time.
\item Proportional harvesting $h(x) = Ex$.
\item Holling type II harvesting $h(x)= \frac{qEx}{m_1E+m_2x}. $
\end{itemize}
Where x is the population that presents the harvesting (prey or predator).For
example, Gupta et al worked with a model with Holling type II harvesting
in prey in \cite{gupta2013bifurcation} and Holling type II harvesting in predator in \cite{gupta2015dynamical}. \par 
The Leslie-Gower term is a formulation  where predator population has logistic growth \cite{leslie1960properties}:
 $$hY\left( 1- \frac{Y}{\alpha X}\right), \quad \alpha= h/n,$$
 but the carrying $C = \alpha X$ is proportional to prey abundance. The term $Y/\alpha X$
is called the Leslie-Gower term \cite{aziz2003boundedness}. Some authors had added a constant to the denominator of Leslie-Gower term, using $Y /(\alpha X+b)$ to avoid singularities when
$X = 0$. This term is called modified Leslie-Gower term.

In \cite{gupta2013bifurcation} the authors studied the following predator prey model of form:
\begin{align}
\frac{dx_1}{dt} &= rx_1 \left( 1- \frac{x_1}{k} \right)- \frac{a_1x_1x_2}{n_1+x_1} - \frac{qEx_1}{m_1E+m_2x_2}, \nonumber \\
\frac{dx_2}{dt} &= sx_2 \left( 1- \frac{a_2x_2}{n+x_1} \right). \label{ec1}
\end{align}

Where $x_1, x_2$ are the population of prey and predators respectively. The
biological assumptions on model \eqref{ec1} are:
\begin{enumerate}
\item Without predator population, the preys have logistic growth $rx_1 \left( 1- \frac{x_1}{k} \right)$ with $r$ the intrinsic growth rate and k, the carrying capacity of environment.
\item $\frac{a_1 x_1}{n_1+x_1}$ is the functional response of Holling type II, $a_1$ and $n_1$ stand for
the predator capturing rate and half saturation constant respectively.
\item The prey presents nonlinear harvesting.
\item The predator has a modified Leslie Gower growth.
\end{enumerate}

Huang et al in \cite{huang2014bifurcations} proposed a Leslie Gower model with Holling type III functional response given by:
\begin{align*}
\frac{dx}{dt} &= rx \left(1- \frac{x}{k} \right) - \frac{mx^2y}{ax^2+bx+1}, \\
\frac{dy}{dt} &= sy \left( 1 - \frac{y}{hx} \right).
\end{align*}
Where $$ p(x)= \frac{mx^2}{ax^2+bx+1}, $$ 
is the Holling type III functional response. To have a biologically meaningful
interpretation we need $p(x) > 0$ (see \cite{buffoni2016dynamics}), thus $b > -2 \sqrt{a}$
 (then, $ax^2+bx+1 =0$  is positive for all $x\geq 0 $).\par 
Based on the work of \cite{gupta2013bifurcation} and \cite{huang2014bifurcations} we propose a model with same assumptions as model \eqref{ec1}, but with functional response of Holling type III. The
model is:
\begin{align}
\frac{dx_1}{dt} &= rx_1 \left( 1- \frac{x_1}{k} \right) - \frac{\bar{m}{x_1}^2 x_2}{a_1x_1^2+b_1x_1+1} - \frac{qEx_1}{m_1E+m_2x_1}, \nonumber \\
\frac{dx_2}{dt} &= sx_2 \left( 1- \frac{a_2x_2}{n+x_1} \right).
\end{align}
Where $x_1, x_2$ are population of prey and predator respectively; all parameters
are positive except $b$, which is arbitrary and $ax_1^2
+ bx1 + 1 > 0, \forall x \geq 0$.\par 

The present paper is divided as follows: in section 2 the positivity  and boundedness of solutions is proved; section 3 has an analysis of the existence and positivity of trivial and interior equilibria points and section 4 shows results about stability of trivial equilibria points obtained in section 3. Finally in section 5, we analyse the stability of interior equilibria when the parameters vary, via the Hopf and Bogdanov-Takens bifurcation. Some numerical simulations are given in this section to show our results.

\section{Basic properties}
Before starting with the mathematical analysis of the model, we set $x_1(t)=kx(t), x_2(t)=ry(t)/ \bar{m}k,$ $\tau = rt$. Applying this change of variable, and using $t$ instead of $\tau$ for simplicity, we have:
\begin{align}
\frac{dx}{dt} &= x(1-x)- \frac{x^2y}{ax^2+bx+1}- \frac{hx}{c+x}, \nonumber \\
\frac{dy}{dt} &= y \left( \delta- \frac{\eta y}{m+x} \right), \label{ec2} \\ 
x(0) &= x_0>0, y(0)=y_0>0. \nonumber
\end{align}
Where $a=a_1k^2$, $b=b_1k$, $h=qE/rm_2k$, $c=m_1E/m_2k$, $\delta=s/r$, $\eta= sa_2/ \bar{m}k^2$, $m=n/k$. $b$ is an arbitrary constant, other new parameters are positive and $ax^2+bx+1 >0, \forall x \geq 0$. 
For now on, we will work with model \eqref{ec2}. \par 
To prove positivity and boundedness we use a lemma taken from \cite{chen2005nonlinear}.

\begin{lemma}
If $a,b>0$ and $\frac{dx}{dt} \leq x(a-bx)$, with $x(0)>0$, $y(0)>0$, then for all $t>0:$
$$x(t)\leq \frac{a}{b-Ce^{-at}}, \quad C=b- \frac{a}{x(0)}. $$ \label{lemma1}
\end{lemma}

\begin{theorem}
Let the initial conditions $x(0) = x_0 > 0, y(0) = y_0 > 0$, then
all solutions of system \eqref{ec2} are positive and bounded.
\end{theorem}

\begin{proof}
Let $x_0, y_0$ be positive and $x(t), y(t)$ the solution of \eqref{ec2}. If $x(t_1) = 0$ for a $t_1 > 0$ then we have from system \eqref{ec2}
$$  \frac{dx}{dt} (t_1)=0, \quad \frac{dy}{dt}(t_1)=0. $$
Then the sets $\{ (0, y), y > 0\}$ and $\{(x, 0), x > 0\}$ are invariant under system \eqref{ec2},
and whenever the solution $(x(t), y(t))$ touches the x-axis or y-axis it will
remain constant and never crosses the axis and the solutions are always in the first
quadrant under positive initial conditions. \par 
Using the positivity of $x$ and $y$, it is not difficult to see that
$$ \frac{dx}{dt} < x(1-x), $$
then, applying Lemma \eqref{lemma1}
$$  x(t) \leq \frac{1}{1-Ce^{-t}}, \quad C=1- \frac{1}{x(0)}. $$
Note that $C > 0$ iff $x(0) > 1$ and $C \leq 0$ if $x(0)\leq 1$. From the fact that
$0 < e^{-t} < 1$ for $t > 0$, we have: if $C > 0$ then $x(t)\leq x(0)$; if $C \leq 0 $ then
$x(t) \leq  1$. Therefore $x(t) \leq \max \{x(0), 1\}$. \par 
Let $M:= \max \{x(0),1 \}$. From second equation of \eqref{ec2}
$$ \frac{dy}{dt} = y \left( \delta - \frac{\eta y}{m+M} y \right)= y(\delta-C_2y), $$
with $C_2= \frac{\eta}{m+M}$. Using Lemma \eqref{lemma1}:
$$ y(t) \leq \frac{\delta}{C_2-C_3e^{-\delta t}}. $$
Again, from the fact $ 0< e^{- \delta t}<1 $ for $t>0$, we have 
$$ y(t) \leq \max  \{ y(0), \frac{\delta(m+M)}{\eta} \}. $$
This completes the proof.

\end{proof}

\section{Existence and stability of equilibria points}

	\subsection{Existence}

To obtain the equilibria solutions of system \eqref{ec2} we look for solutions of the
following system of equations:
\begin{align}
 x(1-x)- \frac{x^2y}{ax^2+bx+1}- \frac{hx}{c+x}&=0 \\
 y \left( \delta- \frac{\eta y}{m+x} \right) &=0. \label{ec12}
\end{align}

From equations above, the isoclines of $y'=0$ are the curves $y = 0$, $ y = \delta(m+x)/\eta$, while the isoclines for $x'=0$ are given by $ x = 0$ and
\begin{equation}
y= (ax^2+bx+1) \left( \frac{-x^2+(1-c)x+(c-h)}{x(c+x)} \right)=p(x)G(x). \label{ec3}
\end{equation}
With $p(x)=ax^2+bx+1>0$. We are interested only in the existence of equilibria points with $x \geq 0$ and $y \geq 0$.
\begin{figure}
\begin{center}

\includegraphics[scale=0.5]{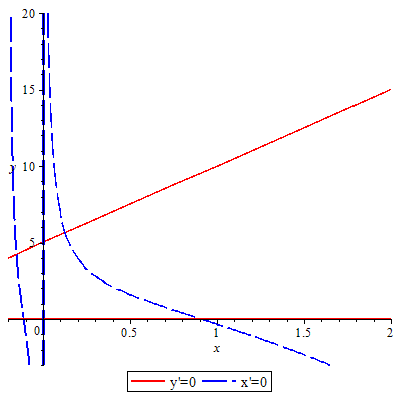}
\end{center}
\caption{Isoclines of the model \eqref{ec2} for parameters: $a = 1, b = 2, c = 0.2, h = 0.1, \delta = 0.5, \eta = 0.1, m = 1 $ } \label{fig1}
\end{figure}
Figure \eqref{fig1} shows the isoclines. It is not difficult to show that, $(0,0)$ is the trivial equilibrium and
$(0, \delta m/ \eta)$ is the unique prey extinction equilibrium, whenever $m\neq 0$.
Moreover, when $y = 0$ and $x > 0$, we have from \eqref{ec3}:
\begin{equation}
x^{\pm}= \frac{1-c \pm \sqrt{(c-1)^2-4(h-c)}}{2}, \label{ec5}
\end{equation}
From previous analysis we have the next theorem.
\begin{theorem}
Let $E^+ = (x^+, 0)$ and $E^- = (x^-, 0)$.  System
\eqref{ec2} has a trivial equilibrium $E = (0, 0)$ and a prey extinction equilibrium
$E_y = (0, \delta m/ \eta)$ (whenever $m \neq 0$). Also, the following assumptions about predator free equilibria
holds:
\begin{itemize}
 \item If $h-c<0$, then $x^{+}>0$ and $x^{-}<0$, so there exists a single positive equilibrium $E^+$.
 \item If $h - c > 0$, $(c - 1)^2 - 4(h - c) > 0$ and $c - 1 < 0$, $x^{+},x^{-}>0$, so there exists two positive
predator free equilibria : $E^+$ and $E^-$.
\item If $ h-c = 0$ and $c-1 < 0$, then $x^{-}=0, x^{+}>0$, so there exists a unique predator free  equilibrium $(1-c, 0)$.
\end{itemize} \label{teo5}
\end{theorem}

When $x \neq 0 \neq y$, then we can have  internal equilibria points, given by  $E^*=(x^*,y^*)$, where $x^{*}$ is a root of 
\begin{equation}
P(x)=x^4+Ax^3+Bx^2+Cx+D=0, \label{ec4}
\end{equation} 
with
\begin{align*}
A&= (c-1)+ \frac{b}{a} + \frac{\delta}{\eta}, \\
B&= (h-c)+ \frac{b}{a} (c-1) + \frac{\delta}{a \eta} (c+m) + \frac{1}{a}, \\
C&= \frac{b}{a}(h-c)+ \frac{1}{a}(c-1)+ \frac{c \delta m}{a \eta}, \\
D&= \frac{h-c}{a},
\end{align*} 
and \begin{equation}
y^{*}= \frac{\delta(m+x)}{\eta}. \label{ec6}
\end{equation}
\begin{figure}
\begin{center}
\includegraphics[scale=0.5]{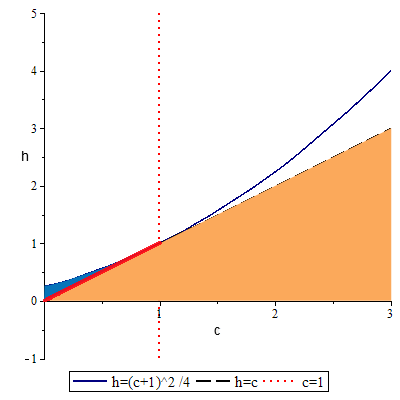}
\caption{Localization of areas $K_1, K_2, K_3$ in the $c-h$ plane. $K1$: blue, $K_2$: solid red line, $K_3$: orange} \label{fig2}
\end{center}
\end{figure}

Equation \eqref{ec4} has four roots, real or complex, but we are interested only in the positive ones. Note that the positive equilibria points are the interception of function \eqref{ec3} with the line $y= \frac{\delta (m+x)}{\eta}$ in the first quadrant (see figure \eqref{fig1}), so we ask for $p(x)G(x) > 0$ in an interval
$(x_1, x_2)$, with $x_1> x_2 \geq 0$; moreover, due to $p(x) > 0$ and $x(x+c) > 0$
for $x \geq 0$, we need $f(x)=-x^2+(1-c)x+(c-h)>0$, for some interval $(x_1, x_2)$.
The roots of $f(x)$ are $x^\pm$ from \eqref{ec5}, it is not difficult to
show that $f(x)$ takes positive values in the first quadrant if and only if
the roots $x^\pm$ are not equal and at least one of them is positive. Using this analysis, we conclude that positive non-trivial equilibria points exist only in one of the following
three areas (see figure \eqref{fig2} ):
\begin{align*}
K_1&= \{(h,c)>0; h>c>0 \hspace{0.1in} c<1 \hspace{0.1in} \text{and}  \hspace{0.1in} h<(c+1)^2 /4 \},\\
K_2&= \{(h,c)>0; 1>c=h>0 \hspace{0.1in}   \}, \\
K_3&= \{(h,c)>0; 0<h<c  \}.
\end{align*}

The easiest case of analysis of equilibria is when equation \eqref{ec4} is reduced to a cubic.
\begin{theorem}
Let $(h,c) \in K_2$ and $y^*$ as \eqref{ec6}. Define:
\begin{align*}
P= B- \frac{A^2}{3}, \quad Q= \frac{2A^3}{27}- \frac{AB}{3}+C, \quad \Delta= \left( \frac{Q}{2}\right)^2 + \left( \frac{P}{3}\right)^3,
\end{align*}
then the following assumptions hold for the existence of equilibria points of system \eqref{ec2}.
\begin{enumerate}
\item When $\Delta>0$, system has a unique equilibrium which is positive if and only if $C<0$, given by $(x^{*},y^{*})$ where:
\begin{equation}
x^*= \sqrt[3]{\frac{-Q}{2} + \sqrt{\Delta} }- \sqrt[3]{\frac{Q}{2}+ \sqrt{\Delta}} - \frac{A}{3} . \label{ec7}
\end{equation}
\item When $\Delta=0,$ system  has two equilibria, $E_1=(x_1,y_1)$, $E_2=(x_2,y_2)$, where:
\begin{equation}
 x_1 = 2 \sqrt[3]{\frac{-Q}{2}}- \frac{A}{3}, \quad x_2=-\sqrt[3]{\frac{-Q}{2}}- \frac{A}{3}, \label{ec8}
\end{equation}
and $y_i$ is the substitution of $x_i$ in $y^{*}$. $E_1$ is positive if and only if $C> \frac{A^{3}-4AB}{12} $ and $E_2$ is positive if and only if $C> \frac{AB}{3}. $
\item When $\Delta<0$, system has three equilibria points (not necessarily positive), $E_k=(x_k,y_k),$ where 
$$x_k= 2 \sqrt{- \frac{P}{3}}  \cos \left( \frac{\phi + 2\pi k}{3} \right)- \frac{A}{3}, $$
and $\phi$ is determined by $ \cos \phi= - \dfrac{Q/2}{\sqrt{-(P/3)^3}} $.
\end{enumerate} \label{teo4}
\end{theorem}

\begin{proof}
In section $K_2$, equation \eqref{ec4} is reduced to
$$x^3+Ax^2+Bx+C=0.$$
Using the Cardano's formula ( \cite{uspensky2004teoria} ) in equation above, we have the following:
\begin{enumerate}
\item When $\Delta>0$, the equation has a real root given by \eqref{ec7} and two complex conjugate. Let $x_1,x_2,x_3$  the roots, and assume (without loss of generality) that $x_1$ is the real one, then $x^3+Ax^2+Bx+C=(x-x_1)(x-x_2)(x-x_3)$, so $-x_1x_2x_3=C$. Due to $x_2 x_3>0$, we arrive to $x_1=-C$, therefore $x_1>0$ iff $C<0$.
\item If $\Delta=0$, we have three real roots, two of them equal, both given by \eqref{ec8}. Substituting the value of $Q$ in \eqref{ec8} we obtain that $x_1>0$ is equivalent to $C> \frac{A^3-4AB}{12} $ and $x_2>0$ is equivalent to $C> \frac{AB}{3} $.
\item For $\Delta<0$, a direct application of Cardano's formula gives the result.
\end{enumerate}
\end{proof}

\begin{figure}
\subfigure[] {\includegraphics[scale=0.5]{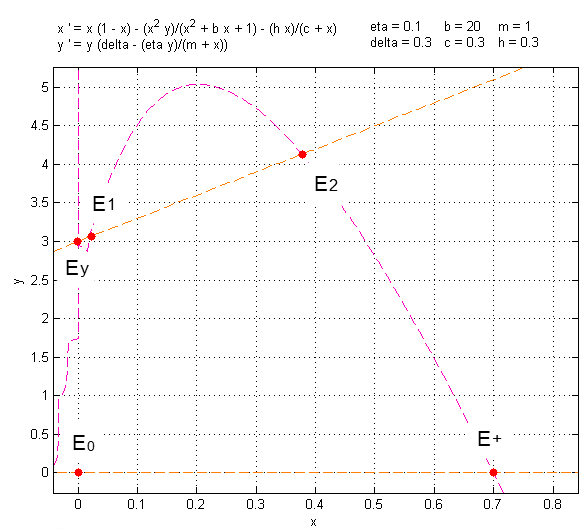}}
\subfigure[] {\includegraphics[scale=0.5]{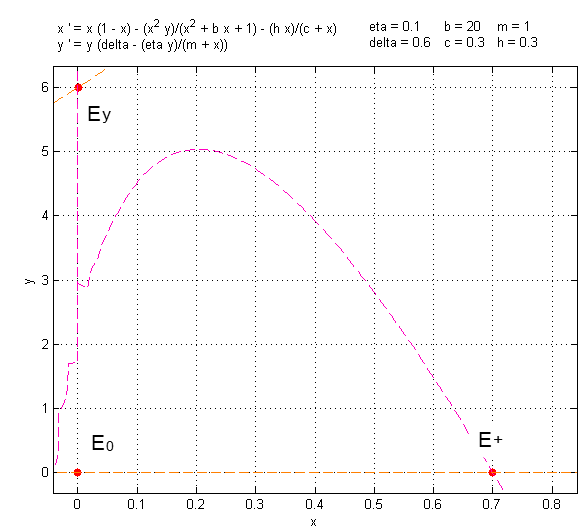}}
\begin{center}
\subfigure[] {\includegraphics[scale=0.5]{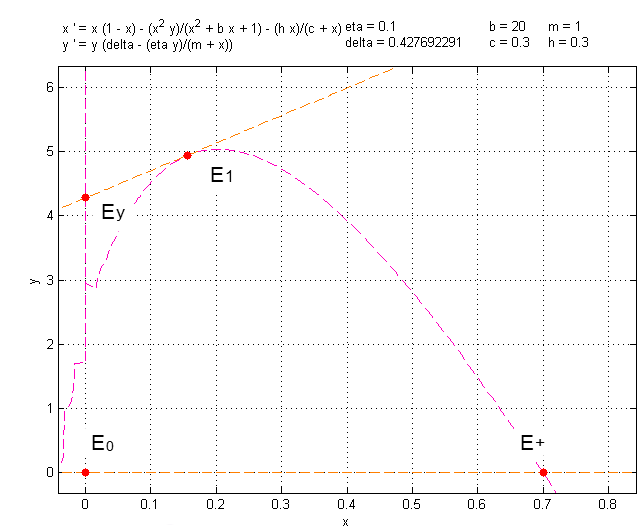}} 
\end{center}
\caption{Isoclines of system \eqref{ec2} in section $K_2$. The values of parameters are: $a=1,b=20, c=h=0.3,m=1,\eta=0.1 $, for these values $\Delta=0$ at $\delta= 0.42769229198509138494$ and $\delta= 10.499994907318960196$. We have the trivial equilibrium $E_0$, the prey extinction $E_y$ and a single predator free $E^+$ (theorem \eqref{teo5}). a) $\Delta<0$: there exists two positive equilibria and a negative one (which is outside the range of figure). b)$\Delta>0$: we have $C>0$, so there is not a positive equilibria point. c) $\Delta=0$, we have a positive equilibrium and a negative one (outside the range of figure) }\label{fig3}
\end{figure}
Figure \eqref{fig3} shows the equilibria points in section $K_2$ depending on the sign of $\Delta$.
In case $K_1$ and $K_3$, we follow the method of Ferrari from \cite{uspensky2004teoria} to solve quartic polynomials (see  appendix \eqref{apen1} ).
\begin{theorem}
Let $(h,c) \in K_1$. Define  $E_1^{\pm}=(x_1^{\pm},y_1^{\pm}) $, $E_2^{\pm}=(x_2^{\pm},y_2^{\pm}) $, where $x_i^{\pm}$ are:
\begin{equation}
x_1^{\pm} = \frac{1}{2} \left( - \sqrt{2u} \pm \sqrt{\Delta_1}- \frac{A}{2} \right), \quad
x_2^\pm = \frac{1}{2} \left( \sqrt{2u} \pm \sqrt{\Delta_2}- \frac{A}{2} \right), \label{ec9}
\end{equation} 
$y_i^{\pm}$ is the substitution on $x_i^{\pm}$ in \eqref{ec6} and the terms $u, \Delta_1, \Delta_2$ are defined in  appendix \eqref{apen1}. Assume $Q_2= A^{3}/8-AB/2+C \neq 0 $, then the following assumptions hold:
\begin{enumerate}
\item If $\Delta_1, \Delta_2<0$, there are no positive equilibrium points.
\item If $\Delta_1 \geq 0, \Delta_2<0$, there are two  equilibria points: $E_1^{-}$ and $E_1^{+}$. $E_1^{-}$ is positive if and only if $u<\frac{1}{2} \left( - \frac{A}{2}- \sqrt{\Delta_1} \right)^2 $ and $E_1^{+}>0$ if and only if $u< \frac{1}{2} \left( - \frac{A}{2}+ \sqrt{\Delta_1} \right)^2$.
\item If $\Delta_2 \geq 0, \Delta_1<0$, there are two  equilibria: $E_2^{-}$ and $E_2^{+}$. $E_2^{-}>0$ if and only if $u>\frac{1}{2} \left(  \frac{A}{2}+ \sqrt{\Delta_2} \right)^2$ and $E_2^{+}>0$ if and only if $u>\frac{1}{2} \left(\frac{A}{2}- \sqrt{\Delta_2} \right)^2$.
\item If $\Delta_2 \geq 0, \Delta_1\geq 0$, we have four  equilibria : $E_1^{\pm}$ and $E_2^{\pm}$.
\end{enumerate}
\end{theorem}
\begin{figure}
\begin{center}
\subfigure[] {\includegraphics[scale=0.5]{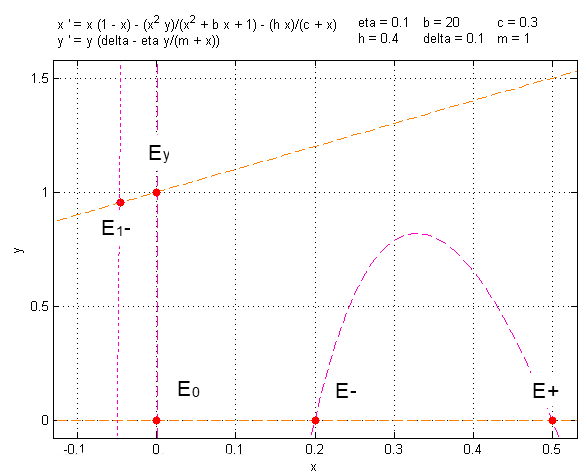}}
\subfigure[] {\includegraphics[scale=0.5]{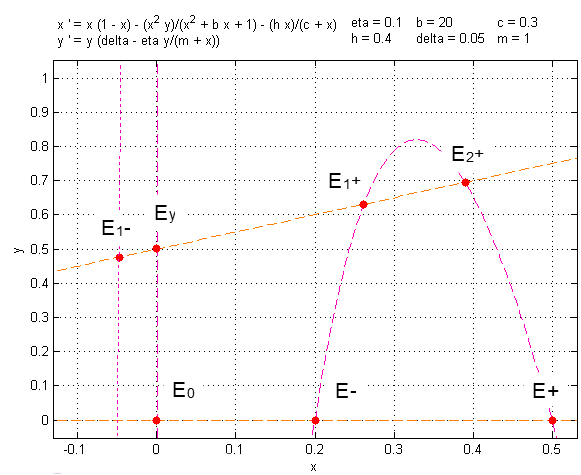}}
\caption{Isoclines of system \eqref{ec2} in section $K_1$. The values of parameters are: $a=1,b=20, c=0.3,m=1, h=0.4 $. We have the trivial equilibrium $E_0$, the prey extinction $E_y$ and two  predator free equilibria: $E^-, E^+$ (theorem \eqref{teo5}). a) $\Delta_1>0, \Delta_2<0$: there are no positive equilibria. b)$\Delta_1>0, \Delta_2>0$: there are two positive and two negative equilibria ($E_2^{-}$ and $E_1^{-}$ which is outside the range of figure).}\label{fig4}
\end{center}
\end{figure}
Figure \eqref{fig4} shows the equilibria points for parameters in $K_1$. \par 
   
The proof of this theorem is directly from the Ferrari's formulas. These formulas can be applied also to section $K_3$, to obtain the following result:

\begin{theorem}
Let $(h,c) \in K_3$, $x_1^{\pm}$ and $x_2^{\pm}$ defined as \eqref{ec9} and $\Delta_1, \Delta_2$ defined as in appendix \eqref{apen1}. Assume $Q_2 \neq 0$ and $E_i^{\pm}$ defined as in previous theorem, then:
\begin{enumerate}
\item If $\Delta_1<0$ (this implies $\Delta_2 \geq 0 $), then we have a unique positive equilibrium $E_2^{+}$.
\item If $\Delta_2<0$ (this implies $\Delta_1 \geq 0 $), then we have a unique positive equilibrium $E_1^{+}$.
\item If $\Delta_1, \Delta_2 \geq 0$ then we have one or three positive equilibria points.
\end{enumerate}
\end{theorem}

\begin{proof}
We know that there are four possible equilibria points $E_i^{\pm}$, with $i=1,2$ and $x_i^{\pm}$ the roots of \eqref{ec4}, which can be four real roots, two real and two complex or four complex (two pairs of complex conjugate). Polynomial \eqref{ec4} can be expressed as:
\begin{align*}
P&(x) = (x-x_1^{+})(x-x_1^{-})(x-x_2^{+})(x-x_2^{-}) , \\
&= x^{4}-(x_1^{+}+x_2^{+}+x_1^{-}+x_2^{-})x^3+(x_1^{+}(x_1^{-}+x_2^{+}+x_2^{-})+ x_1^{-}(x_2^{+}+x_2^{-}) \\
&+x_2^{+}x_2^{-} )x^{2}-(x_1^{+}x_1^{-}x_2^{+}+ x_1^{+}x_1^{-}x_2^{-}+ x_1^{+}x_2^{+} x_2^{-} + x_1^{-}x_2^{+}x_2^{-})x + x_1^{+}x_1^{-}x_2^{+}x_2^{-},
\end{align*}
so, $D=x_1^{+}x_1^{-}x_2^{+}x_2^{-}= \frac{h-c}{a}<0$. Note that there is no root equal zero due to the sign of $D$. Making an analysis of the possibilities in roots, it is not difficult to show that for $h-c<0$ we have three possible cases: two complex and two real with different sign, three positive and one negative or three negative and one positive.
\begin{enumerate}
\item If $\Delta_1<0$, then $x_1^{\pm}$ are both complex conjugate. Then the roots $x_2^{\pm}$ are real with different sign, moreover $x_2^{-}<0<x_2^{+}$. This implies $\Delta_2\geq 0$. Therefore the positive equilibrium is $E_2^{+}$.
\item If $\Delta_2<0$, then $x_2^{\pm}$ are both complex conjugate. Then the roots $x_1^{\pm}$ are real with different sign, moreover $x_2^{-}<0<x_2^{+}$. This implies $\Delta_2\geq 0$. Therefore the positive equilibrium is $E_1^{+}$.
\item If $\Delta_1 , \Delta_2 \geq0$, then the roots are three positive and one negative or one positive and three negative.
\end{enumerate}  
\end{proof}
\begin{figure}
\subfigure[] {\includegraphics[scale=0.5]{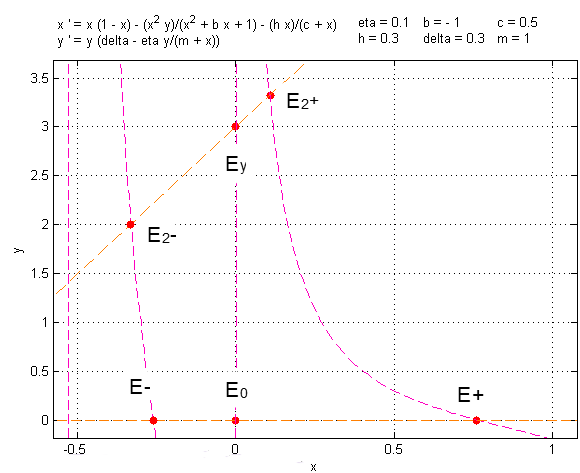}}
\subfigure[] {\includegraphics[scale=0.5]{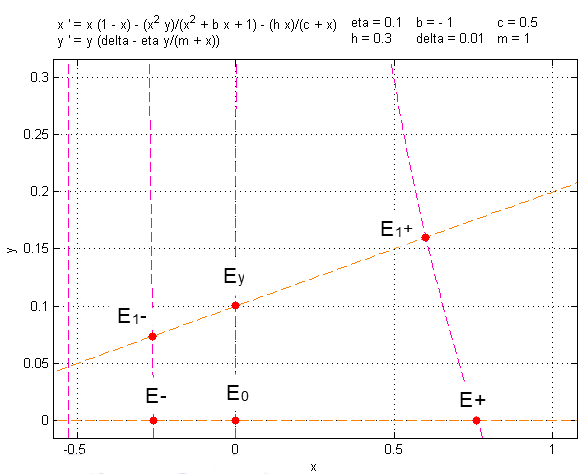}}
\begin{center}
\subfigure[] {\includegraphics[scale=0.5]{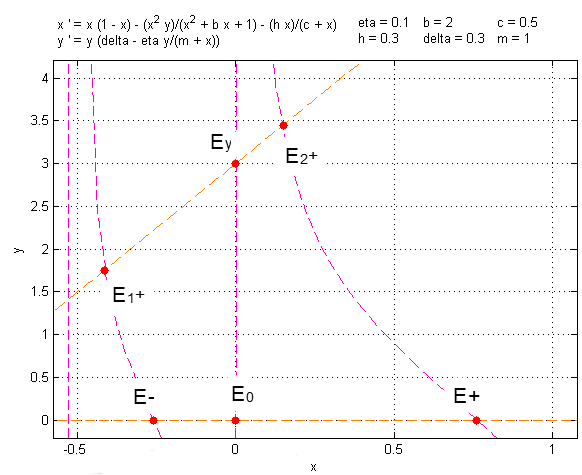}} 
\end{center}
\caption{Isoclines of system \eqref{ec2} in section $K_2$. The values of parameters are: $a=1,b=-1, c=0.5, h=0.3, m=1 $. We have the trivial equilibrium $E_0$, the prey extinction $E_y$ and a positive predator free equilibrium $E^+$ (theorem \eqref{teo5}). a) $\Delta_1<0, \Delta_2>0$: there is a positive equilibrium $E_2^+$ and a negative one $E_2^{-}$. b)$\Delta_1<0, \Delta_2>0$: there is a positive equilibrium $E_1^+$ and a negative one $E_1^{-}$. c) Changing $b=2$, then $\Delta_1>0, \Delta_2>0$, we have a positive equilibrium $E_1^{+}$ and three negative ones (two of them outside the range of figure). }\label{fig5}
\end{figure}
		
	\subsection{Stability}

From  theorem 2, we have four trivial equilibria points, $E=(0,0)$, $E_y=(0, \delta m/ \eta)$, $E^{+}=(x^{+},0)$ and $E^{-}=(x^{-},0)$. The stability of each one is given in the following theorems:

\begin{theorem}
The following hold for trivial equilibria point $E$ of system \eqref{ec2}
\begin{itemize}
\item If $c-h>0$, it is an unstable node.
\item When $c-h<0$, it is a saddle. 
\item If $h=c$ and $c \neq 1$ then $E$ is a saddle node, ie, is divided into two parts along the positive and negative $y-$axis, one part is a parabolic sector and the other part consists of two hyperbolic sectors. Moreover, the parabolic sector is on the right half plane if $c<1$ and on the left half plane when $c>1$.
\item If $h=c$ and $c=1$, it is a saddle.
\end{itemize}
\end{theorem}

\begin{proof}
For $E=(0,0)$, the Jacobian matrix is given by 

\begin{equation}
J(E)=\left( 
\begin{matrix}
1-\frac{h}{c} & 0 \\ 0 & \delta
\end{matrix}
\right).
\end{equation}
The characteristic polynomial is $P_E( \lambda )= (1-h/c-\lambda)(\delta- \lambda)$, with roots $ \lambda_1 = 1- h/c $ and $ \lambda_2=\delta $. Clearly,  when $c-h<0$, $\lambda_1<0$ and we have a saddle. When $c-h>0$ we have an unstable node  \par 
When $h=c$ we have an eigenvalue $\lambda=0$, so using theorem 7.1 from \cite{zhi2006qualitative}, we can rewrite the system as
\begin{align*}
\frac{dx}{dt} &= -{\frac {{x}^{2} \left( ac{x}^{2}+a{x}^{3}-a{x}^{2}+bcx+b{x}^{2}-bx+cy
+xy+c+x-1 \right) }{ \left( a{x}^{2}+bx+1 \right)  \left( c+x \right) 
}}, \\
\frac{dy}{dt} &= \delta\,y-{\frac {{y}^{2}\eta}{m+x}}.
\end{align*}
Making the change of time $\tau= \delta t$, and using $t$ instead of $\tau$, then system above is transformed into:
\begin{align*}
\frac{dx}{dt} &=-{\frac {{x}^{2}y}{\delta\, \left( a{x}^{2}+bx+1 \right) }}-{\frac {{x
}^{2} \left( ac{x}^{2}+a{x}^{3}-a{x}^{2}+bcx+b{x}^{2}-bx+c+x-1
 \right) }{\delta\, \left( a{x}^{2}+bx+1 \right)  \left( c+x \right) }
}\\
&= P_2(x,y), \\
\frac{dy}{dt} &= -{\frac {\eta\,{y}^{2}}{\delta\, \left( m+x \right) }}+y= y+Q_2(x,y).
\end{align*}
Taking $\phi(x)=0$, and expanding $\psi:=P_2(x, \phi(x))$:
$$ \psi = -{\frac { \left( c-1 \right) {x}^{2}}{c\delta}}- \left( 1-{\frac {c-1}
{c}} \right) \frac{x^3}{c \delta}.
 $$
 So, $m=2$, $a_m= -(c-1)/c \delta$ and sgn$(a_m)=$sgn$(1-c)$. By theorem 7.1 of \cite{zhi2006qualitative}, if $c \neq 1$ then $E$ is a saddle node. If $c=1$ then we have $m=3$, $a_m=-1$ and $E$ is a saddle.
\end{proof}

\begin{theorem}
The following holds for equilibria $E_y=(0, \delta m/ \eta)$:
\begin{itemize}
\item $E_y$ is locally asymptotically stable when $c-h<0$ and a saddle when $c-h>0$.
\item If $c=h$ and $c\delta\,m+c\eta-\eta \neq 0$, then it is a saddle node. Moreover if $c\delta\,m+c\eta-\eta>0$(<0) the parabolic sector is in the right (left) half-plane.
\item If $c=h$ and $c\delta\,m+c\eta-\eta = 0$, then $E_y$ is an unstable node if $ b\delta\eta m-{\delta}^{2}{m}^{2}-2 \delta \eta
m-\delta \eta-{\eta}^{2}<0$ and a saddle if $ b\delta \eta m-{\delta}^{2}{m}^{2}-2 \delta\eta m-\delta \eta-{\eta}^{2}>0$.
\end{itemize}
\end{theorem}
\begin{proof}
The Jacobian matrix at this point is given by
\begin{equation}
J(E_y)= \left( 
\begin{matrix}
1- \frac{h}{c} &0 \\ \frac{\delta^{2}}{\eta} & - \delta
\end{matrix}
\right).
\end{equation}
The polynomial is given by $P_{E_y}(\lambda)= (1-h/c- \lambda)(-\delta- \lambda)$, with roots $\lambda_1 = 1-h/c$ and $\lambda_2= - \delta<0$.  $\lambda_1<0$ when $c-h<0$ (stable node) and $\lambda_1>0$ for $c-h>0$ (saddle). Moreover, when $h=c$ we can make a change of coordinates $u=x, v=y+\delta m/ \eta  $, obtaining:
\begin{align}
\frac{du}{dt} &= u \left( 1-u \right) - \dfrac{{u}^{2} \left( v+{\frac {\delta\,m}{\eta}}
 \right)}{\left( a{u}^{2}+bu+1 \right)} -{\frac {hu}{c+u}}, \\
 \frac{dv}{dt} &= {\frac { \left( \delta\,m+\eta\,v \right)  \left( \delta\,u-\eta\,v
 \right) }{\eta\, \left( m+u \right) }}. \label{ec17}
\end{align}
Let $X=(u,v)$, and $F$ the right hand side of system above, then it can be written as $X'=F(X)= J(E_y)X+ G(X)$ with $G(X)=F(X)- J(E_y)X$.
We make the change of coordinates $u= \eta Y_1/ \delta$, $v= Y_1+Y_2$ to obtain:
\begin{align*}
Y_1'&= -{\frac {{{Y_1}}^{2} \left( c\delta\,m+c\eta-\eta \right) }{c\delta
}} +O(Y^3), \\
Y_2' &= -\delta\,{ Y_2}+{\frac {{{ Y_1}}^{2} \left( c\delta\,m+c\eta-\eta
 \right) }{c\delta}}-{\frac {\eta\,{{Y_2}}^{2}}{m}} + O(Y^3),
\end{align*}
where $O(Y^3)$ contains all the terms of the form $a_{ij} Y_1^{i}Y_2^{j}$, with $i+j \geq 3$. Making a change in time by $\tau= - \delta t$ and using $t$ instead of $\tau$ for simplicity, we have:

\begin{align}
Y_1' &= {\frac {{Y_{{1}}}^{2} \left( c\delta\,m+c\eta-\eta \right) }{{\delta}^
{2}c}}+  O(Y^3)=P_2(Y_1,Y_2), \nonumber \\
Y_2' &= Y_{{2}}+{\frac {\eta\,{Y_{{2}}}^{2}}{\delta\,m}}-{\frac {{Y_{{1}}}^{2}
c\delta\,{m}^{2}+{Y_{{1}}}^{2}c\eta\,m-{Y_{{1}}}^{2}\eta\,m}{{\delta}^
{2}cm}}+ O(Y^3)=Y_2+Q_2(Y_1,Y_2). \label{ec18}
\end{align} 
Note that the first terms of $P_2$ does not include $Y_2$, so if $\phi(Y_1)$ satisfies $\phi+Q_2(Y_1, \phi)=0$, then $$P_2(Y_1, \phi)= {\frac {{Y_{{1}}}^{2} \left( c\delta\,m+c\eta-\eta \right) }{{\delta}^
{2}c}}+  O(Y_1^3).$$
 Due to $m=2,$ if $c\delta\,m+c\eta-\eta\neq 0$ then  $E_y$ is a saddle node. \par 
 When $c\delta\,m+c\eta-\eta= 0$, then system \eqref{ec18} becomes:
 \begin{align*}
 Y_1' &= -{\frac { \left( b\delta\,\eta\,m-{\delta}^{2}{m}^{2}-2\,\delta\,\eta
\,m-\delta\,\eta-{\eta}^{2} \right) {Y_{{1}}}^{3}}{{\delta}^{3}}}+{
\frac {{Y_{{1}}}^{2}\eta\,Y_{{2}}}{{\delta}^{2}}} +O(Y^4)=P_3(Y_1,Y_2), \\
Y_2' &=Y_2+ \left( -{\frac {{\eta}^{2}Y_{{1}}}{{\delta}^{2}{m}^{2}}}+{\frac {\eta
}{\delta\,m}} \right) {Y_{{2}}}^{2}+ \left( -{\frac {{Y_{{1}}}^{2} Y_2\eta
}{{\delta}^{2}}} \right)\\
& +{\frac { \left( b\delta\,\eta\,{m}^
{3}-{\delta}^{2}{m}^{4}-2\,\delta\,\eta\,{m}^{3}-\delta\,\eta\,{m}^{2}
-{\eta}^{2}{m}^{2} \right) {Y_{{1}}}^{3}}{{m}^{2}{\delta}^{3}}}+O(Y^4),=Q_3(Y_1,Y_2). 
 \end{align*}
 Again, if $\phi=O(Y_1)^2$ then $P_3(Y_1, \phi)= -{\frac { \left( b\delta\,\eta\,m-{\delta}^{2}{m}^{2}-2\,\delta\,\eta
\,m-\delta\,\eta-{\eta}^{2} \right) {Y_{{1}}}^{3}}{{\delta}^{3}}}+ O(Y_1^4)$, so $E_y$ is an unstable node if $ b\delta\eta m-{\delta}^{2}{m}^{2}-2 \delta \eta
m-\delta \eta-{\eta}^{2}<0$ and a saddle if $ b\delta \eta m-{\delta}^{2}{m}^{2}-2 \delta\eta m-\delta \eta-{\eta}^{2}>0$.
  \end{proof}

\begin{theorem}
Whenever $E^\pm$ exists and its component $x^+ $ ($x^-$) is positive, then it is unstable
\end{theorem}
\begin{proof}
The Jacobian matrix for this case is $$ J(E^{\pm}) = \left( \begin{matrix}
1-2x^\pm - \frac{hc}{(c+x^{\pm})^{2}} & - \frac{(x^{\pm})^{2}}{a (x^{\pm})^2+bx^{\pm}+1} \\ 0 & \delta
\end{matrix}
\right), $$
with eigenvalues $\lambda= \delta>0, $ and $\lambda_2=1- 2x^{\pm}- \frac{hc}{(c+x)^2}$. So $E^{\pm}$ is always unstable.
\end{proof}

Due to the multiple cases that we have for the  existence of interior equilibria points, the analysis of stability via linearization of each one, will be extensive and complicated.
In further sections we will not focus our attention in stability analysis of all interior equilibria points in the $K_i$'s, instead of, our goal is to find the critical values of parameters that let the model to present bifurcations, and then, make an analysis  for  parameters near to critical point, in order to obtain a view of the phase plane of system around them.

\section{Bifurcation analysis}

	\subsection{Hopf bifurcation}
	
In previous section we have seen the existence of multiple equilibria points. One of the cases of interest is the existence of a Hopf bifurcation, this happens when an equilibrium   changes its stability letting the existence of a limit cycle around it. The Hopf bifurcation occurs when the Jacobian matrix has at an equilibrium $E^{*}$,  a pair of pure imaginary eigenvalues, ie, $Tr(J(E^*))=0$ and $\det (J(E^{*}))>0$. \par 

Let $E^{*}=(x^{*},y^{*})$ an equilibrium of system,  then its Jacobian matrix is given by

\begin{equation}
J(E^{*}) = \left( 
\begin{matrix}
\alpha_{10} & \alpha_{01} \\ \beta_{10} & \beta_{01}
\end{matrix}
\right), \label{ec10} \end{equation}
where:
\begin{align*}
\alpha_{10} &= 1-2x^{*}- \dfrac{x^{*}y^{*}(bx^{*}+2)}{(a(x^{*})^{2} + bx^{*}+1)^{2}} - \dfrac{hc}{(c+x^{*})^{2}} , \\
\alpha_{01} &=  -\dfrac{(x^{*})^{2}}{a(x^{*})^{2}+bx^{*}+1}, \\
\beta_{10} &= \dfrac{\delta^{2}}{\eta}, \\
\beta_{01} &= - \delta. 
\end{align*}

To obtain a pair of pure imaginary eigenvalues of $J(E^{*})$ we ask for
\begin{align}
\delta = 1-2x^{*}- \dfrac{x^{*}y^{*}(bx^{*}+2)}{(a(x^{*})^{2} + bx^{*}+1)^{2}} - \dfrac{hc}{(c+x^{*})^{2}} &=: \delta^{H},\\
-( \delta^{H})^2+ \dfrac{(\delta^H)^2 x^2}{\eta(ax^2+bx+1)} &>0.
\end{align}
To ensure the existence of Hopf bifurcation we need to verify the no-degenerate condition 
\begin{equation*}
\frac{d(Tr(E^{*}))}{d(\delta)} \mid_{\delta=\delta^{H}}  =-1 \neq0. 
\end{equation*}

In order to discuss the stability of the limit cycle, we use a change of coordinates $u=x-x^{*}, v=v-v^{*}$ to transform system \eqref{ec2} into
\begin{align*}
\frac{du}{dt} &=  \left( u+{x^{*}} \right)  \left( 1-u-x^{*} \right) -{\frac {
 \left( u+x^{*} \right) ^{2} \left( v+y^{*} \right) }{a \left( u
+x^{*} \right) ^{2}+b \left( u+x^{*} \right) +1}}-{\frac {h
 \left( u+x^{*} \right) }{c+u+x^{*}}}, \\
 \frac{dv}{dt} &= \left( v+y^{*} \right)  \left( \delta-{\frac {\eta\, \left( v+y^{*} \right) }{m+u+x^{*}}} \right). 
\end{align*}
Using the Taylor expansion around $(0,0)$, then system above is rewritten as:
\begin{align}
\frac{du}{dt} &= \alpha_{10} u + \alpha_{01} v+ \alpha_{20}u^{2}+ \alpha_{11}uv+ \alpha_{30} u^{3}+ \alpha_{21} u^{2}v + Q_1(x,y), \nonumber \\
\frac{dv}{dt} &=  \beta_{10}u+ \beta_{01}v+ \beta_{20}u^{2}+ \beta_{11}uv+ \beta_{02}v^{2}+ \beta_{30}u^{3}+ \beta_{21}u^{2}v+ \beta_{12}uv^{2}+ Q_2(x,y), \label{ec19}
\end{align} 
where $\alpha_{10},  \alpha_{01}, \beta_{10}, \beta_{01}$ are given by the Jacobian matrix $J(E^{*})$ in \eqref{ec10}, $Q_1, Q_2$ are polinomials in $x^{i},y^{j}$ with $i+j \geq 4$ and 
\begin{align*}
\alpha_{20} &= -1+ \dfrac{y^{*}(ab(x^{*})^3+3a(x^{*})^2-1)}{(a(x^{*})^2+bx^{*}+1)^3} + \dfrac{hc}{(c+x^{*})^3}, \\
\alpha_{11} &= - \dfrac{x^{*}(bx^{*}+2)}{(a(x^{*})^2+bx^{*}+1)^2}, \\
\alpha_{30} &= - \dfrac{hc}{(c+x^{*})^4}- \dfrac{y^{*}(a(x^{*})^2-1)(ab(x^{*})^2+4ax^{*}+b)}{(a(x^{*})^2+bx^{*}+1)^4}, \\
\alpha_{21} &= \dfrac{ab(x^{*})^3+3a(x^{*})^2-1}{(a(x^{*})^2+bx^{*}+1)^3}, \\
\beta_{20} &= - \dfrac{\delta^2}{\eta(m+x^{*})}, \quad \beta_{11} = \dfrac{2 \delta}{m+x^{*}}, \quad \beta_{02} = - \dfrac{\eta}{m+x^{*}},\\
 \quad \beta_{30} &= \dfrac{\delta^2}{\eta(m+x^{*})^2}, \quad
\beta_{21} = - \dfrac{2 \delta}{(m+x^{*})^2}, \quad \beta_{12} =  \dfrac{\eta}{(m+x^{*})^2}.
\end{align*}
Therefore, using matrix notation, system \eqref{ec19} can be expressed as:
\begin{equation}
\left( \begin{matrix}
\frac{du}{dt} \\ \frac{dv}{dt}
\end{matrix}  \right) = J(E^{*}) \left(  \begin{matrix}
u \\v
\end{matrix}  \right) + G(u,v),
\end{equation}
with $G = \left( \begin{matrix}
\alpha_{20}u^{2}+ \alpha_{11}uv+ \alpha_{30} u^{3}+ \alpha_{21} u^{2}v + Q_1(x,y) \\
\beta_{20}u^{2}+ \beta_{11}uv+ \beta_{02}v^{2}+ \beta_{30}u^{3}+ \beta_{21}u^{2}v+ \beta_{12}uv^{2}+ Q_2(x,y)
\end{matrix} \right) .$ \par 
At $ \delta= \delta^{H} $, matrix $J(E^{*})$ has a pair of pure imaginary eigenvalues, so $\alpha_{10}= \beta_{01} $. Let $\omega= \sqrt{\det(J(E^{*}))}>0$, we make the change of coordinates $ u=Y_2, v=\omega Y_1- \frac{\delta}{a_{12}} Y_2 $ obtaining the following equivalent system:
\begin{align*}
\left( \begin{matrix}
\frac{dY_1}{dt} \\ \frac{dY_2}{dt}
\end{matrix} \right) = \left(  \begin{matrix}
0 & - \omega \\ \omega & 0
\end{matrix} \right)+ \left( \begin{matrix}
f(Y_1,Y_2)+Q_3 \\ g(Y_1, Y_2)+Q_4
\end{matrix} \right),
\end{align*}
with $Q_3, Q_4$ functions in $Y_1^{i} Y_2^{j}$ for $i+j \geq 4$ and
\begin{align*}
f &= \left( - \frac{\alpha_{21} \delta}{\alpha_{01}}+ \alpha_{30}  \right)Y_{2}^{3}+ \frac{\alpha_{21}Y_1 Y_2^{2}}{\alpha_{01}}  + \left( \alpha_{20}- \frac{\alpha_{11} \delta}{\alpha_{01}} \right)Y_2^{2}+ \frac{\alpha_{11} \omega Y_1 Y_2}{\alpha_{01}}\\
g &= \left( - \frac{\beta_{21} \delta}{\alpha_{01}} + \frac{\beta_{12} \delta^{2}}{\alpha_{01}^{2}}+ \beta_{30} \right) Y_2^{3}+ \left( \frac{\beta_{21} \omega}{\alpha_{01}}- \frac{2 \beta_{12} \delta \omega}{\alpha_{01}^2}  \right) Y_1 Y_2^2\\
& + \left( -\frac{\beta_{11}}{\alpha_{01}} + \beta_{20}+ \frac{\beta_{02} \delta^{2}}{ \alpha_{01}^2} \right) Y_2^2+ \frac{\beta_{12} \omega^2 Y_1^2 Y_2}{\alpha_{01}^2} + \left( \frac{\beta_{11} \omega}{\alpha_{01}}- \frac{2 \beta_{02} \omega \delta}{\alpha_{01}^2} \right)Y_1 Y_2 \\
&+  \frac{\beta_{02} \omega_{2} Y_1^2}{\alpha_{01}^2}.
\end{align*}
Using theorem (3.4.2) from \cite{guckenheimer2013nonlinear}, we define the following coefficient:
\begin{align}
 l&:= \frac{\alpha_{21} \omega}{ 8 \alpha_{01 }}+ \frac{\beta_{12} \omega^{2}}{8 \alpha_{01}^2}- \frac{3 \beta_{21} \delta}{8 \alpha_{01}}+ \frac{3 \beta_{12} \delta^{2} }{8 \alpha_{01}^2}+ \frac{3 \beta_{30}}{8} + \nonumber \\
&+\frac{1}{16\omega} \left(  
\frac{\alpha_{11} \omega  }{\alpha_{01}} \left( 2 \alpha_{20}- \frac{2 \alpha_{11} \delta}{\alpha_{01}} \right)- \left( \frac{\beta_{11} \omega}{\alpha_{01}}- \frac{2 \beta_{02} \delta \omega}{\alpha_{01}^2} \right) \left( \frac{2 \beta_{02} \omega^2 }{\alpha_{01}^2}- \frac{2 \beta_{11} \delta}{\alpha_{01}}+ 2 \beta_{20}+ \frac{2 \beta_{02} \delta^{2}}{\alpha_{01}^2} \right)
  \right) + \nonumber \\
 & + \frac{1}{16 \omega} \left(
  \left( 2 \alpha_{20}- \frac{2 \alpha_{11} \delta}{\alpha_{01}} \right) \left( - \frac{2 \beta_{11} \delta}{\alpha_{01}} + 2 \beta_{20} + \frac{2 \beta_{02} \delta^{2}}{\alpha_{01}^2} \right)   \right). \label{ec20} 
\end{align}
\begin{theorem}
Suppose system \eqref{ec2} has an interior equilibrium point $E^{*}$ which satisfies $\delta= \delta^{H}$ and $-( \delta^{H})^2+ \dfrac{(\delta^H)^2 x^2}{\eta(ax^2+bx+1)} >0  $. Assume $l \neq 0$, with $l$ defined in \eqref{ec20},  then system undergoes a Hopf bifurcation around $E^{*}$, which implies the existence of periodic solutions around $E^{*}$. Moreover, the periodic solutions are stable cycles if $l>0$, and repelling if $l<0$.
\end{theorem}
 			
	\subsection{Bogdanov-Takens bifurcation}
	
	When at some vales of the parameters, say $\alpha=(\alpha_1, \alpha_2)$ there exists an equilibrium with two zero eigenvalues (the \textit{Bogdanov-Takens condition}), then for nearby  values of $(\alpha_1, \alpha_2)$ we can expect the appearance of new phase portraits of the system, implying that the Bogdanov-Takens bifurcation of 	codimension two, has occurred. The Bogdanov-Takens condition is equivalent to $Tr(J(E^*))= \det(J(E ^*))=0$, for an equilibrium $E^{*}$. In this section we compute the  Bogdanov-Takens condition in terms of two parameters of the model: $h$ and $ \delta$. Then, we develop the normal form of this bifurcation, computing the non-degeneracy conditions, following the steps given by \cite{kuznetsov2013elements}. Finally we give some examples to sketch the bifurcations curves (using the theoretical results obtained ) and the phase portraits of solutions, in terms of parameters for nearby values.  \par 

\subsubsection{Existence of equilibria points with  double zero eigenvalues}

From section 3.2, none of the trivial equilibria points satisfy the Bogdanov-Takens condition, so we focus our study in interior equilibria points. \par 
From previous section, an interior equilibria point $E^{*}=(x^{*},y^{*}), x^{*} \neq 0 \neq y^{*}$, has a Jacobian matrix  given by \eqref{ec10} and its  characteristic polynomial  is 
$$  \lambda ^{2}- \text{Trace}(J(E^{*})) \lambda  + \det(J(E^{*}))=0. $$
The determinant and trace can be simplified as:
\begin{equation}
\begin{aligned}
\det (J(E^{*}))& = \delta \left( -1 + \dfrac{hc}{(c+x^{*})^{2}}+2x^{*} + \dfrac{x^{*} y^{*}(bx^{*}+2)}{(a(x^{*})^{2} + bx^{*}+1)^{2}} + \dfrac{\delta (x^{*})^{2}}{ \eta (a(x^{*})^{2} + bx^{*}+1) } \right) \\
\text{trace}(J(E^{*})) &= 1-2x^{*}- \dfrac{x^{*}y^{*}(bx^{*}+2)}{(a(x^{*})^{2} + bx^{*}+1)^{2}} - \dfrac{hc}{(c+x^{*})^{2}}- \delta.
\end{aligned}
\end{equation}

For simplicity, we omit the $^{*}$ and refer to an interior equilibrium as $E=(x,y)$. If $E$ has two zero eigenvalues, then $\text{trace}(J(E))=0= \det(J(E))$, so we have following system of algebraic equations: 
\begin{equation}
\begin{aligned}
 -1 + \dfrac{hc}{(c+x)^{2}}+2x + \dfrac{x y(bx+2)}{(ax^{2} + bx+1)^{2}} + \dfrac{\delta x^{2}}{ \eta (ax^{2} + bx+1) } &=0, \\
 1-2x- \dfrac{xy(bx+2)}{(ax^{2} + bx+1)^{2}} - \dfrac{hc}{(c+x)^{2}}- \delta &=0. \label{BT1}
\end{aligned}
\end{equation}
Adding both equations in \eqref{BT1}, we obtain 

\begin{equation}
 \dfrac{\delta x^{2}}{ \eta (ax^{2} + bx+1) }- \delta =0,
\end{equation}
or equivalently, 
\begin{equation}
x^{2}(a \eta-1)+b \eta x+ \eta =0. \label{BT2}
\end{equation}
And from second equation of \eqref{BT1} we have that
\begin{equation}
y= \dfrac{(ax^{2} + bx+1)^{2}}{x(bx+2)} \left(1-2x - \dfrac{hc}{(c+x)^{2}}- \delta  \right). \label{BT3}
\end{equation}

\begin{lemma}
System \eqref{BT1} is equivalent to system \eqref{BT2}-\eqref{BT3} in sense that both systems have the same solutions. \label{BTlem1}
\end{lemma}

Using lemma \eqref{BTlem1}, we analyse the solutions of \eqref{BT2}-\eqref{BT3}. From \eqref{BT2} we can obtain one or two possible values for $x$ (not necessarily positive), and each value of $x$ has a single value of $y$ associated, by the relationship \eqref{BT3}. So, we can have at most, two possible points where BT bifurcation can occur, $(x_1,y_1)$ and $(x_2,y_2)$, where $x_i$ is a root of \eqref{BT2} and $y_i$ is the respective substitution in \eqref{BT3}.
The sign of $x_1,x_2$ depend on the sign of $a \eta-1$, so we analyse three possible cases : $a \eta <1, a \eta>1$ and $a \eta=1$.
\subsubsection{Case $a \eta=1$}

The easiest case is when \textbf{$a \eta=1$}. If $a \eta=1$, the quadratic equation \eqref{BT2} is simplified to a linear one with root \textbf{$x_1=-1/b$}, which is positive iff \textbf{$b<0$}; its respective value of $y$ in \eqref{BT3}  is:

\begin{equation}
y_1 = {\frac {{b}^{3}{c}^{2}\delta-{b}^{3}{c}^{2}+{b}^{3}ch-2\,{b}^{2}{c}^{2
}-2\,{b}^{2}c\delta+2\,{b}^{2}c+4\,bc+b\delta-b-2}{{b}^{4} \left( bc-1
 \right) ^{2}{\eta}^{2}}}. \label{BT4}
\end{equation}
Let $ \alpha= b \delta-b-2 $, then 
\begin{equation}
y_1 = \dfrac{\alpha b^{2}c^{2}+b(b^{2}h-2 \alpha)c+\alpha}{b^{4}(bc-1)^{2}\eta^{2}}. \label{BT5}
\end{equation}
$E_1=(x_1,y_1)$ satisfies $ \text{trace}(J(E_1))=0=\det(E_1) $, but $E_1$ is not necessary an equilibrium, so we ask that $E_1$ satisfy the equations for equilibria points, ie, 

\begin{equation}
\begin{aligned}
1-x- \frac{xy}{ax^{2}+bx+1}- \frac{h}{c+x} &=0, \\
\delta- \frac{\eta y }{m+x} &=0.
\end{aligned}
\end{equation}
 Substituting the value for $a, x_1,y_1$ and simplifying,

\begin{align*}
&1+{\frac {\alpha\,{b}^{2}{c}^{2}+b \left( {b}^{2}h-2\,\alpha \right) c
+\alpha}{{b}^{3} \left( bc-1 \right) ^{2}\eta}}-{\frac {h}{c}}=0, \\
& \delta- \dfrac{\alpha\,{b}^{2}{c}^{2}+b \left( {b}^{2}h-2\,\alpha \right) c+\alpha}{\eta\,{b}^{4} \left( bc-1 \right) ^{2} \left(m- \frac{1}{b} \right)}=0.
\end{align*}
 
From previous equations we can obtain expressions for $h$ and $ \delta $

\begin{align}
h &= {\frac { \left( {b}^{3}\eta+{b}^{2}\eta+b\delta-b-2 \right)  \left( bc
-1 \right) ^{2}}{{b}^{3} \left( {b}^{2}c\eta-b\eta-c \right) }}
=: h_1 , \label{BT8}\\
\delta &= -{\frac {bc-b-2}{b \left( {b}^{4}c\eta\,m-{b}^{3}c\eta-{b}^{3}\eta\,m-
{b}^{2}cm+{b}^{2}\eta+1 \right) }}
\nonumber \\
&=: \delta_1 . \label{BT9} 
\end{align}
Note that the values $\delta_1, h_1$, $y_1$ satisfy the equilibria equations, so $y_1= \delta_1(m- \frac{1}{b} )/\eta$ and it is positive for $b<0$. 
Using the previous results we can enunciate the following:

\begin{theorem}
Let $c, \eta, m>0$, $a= \frac{1}{\eta}$ and $b>-2 \sqrt{a}$.  If $h=h_1, \delta=\delta_1$, where $h_1, \delta_1$ are given by \eqref{BT8}, \eqref{BT9} (whenever they are positive) , then the system has an equilibrium at $(-\frac{1}{b},y_1)$  with a double zero eigenvalue. Moreover, if  $b<0$ then the equilibrium is positive. \label{BTteo1}
\end{theorem}
\begin{proof}
Proof follows from previous analysis.
\end{proof} 
\subsubsection{Case $a \eta<1$}
When $a \eta<1$, equation \eqref{BT2} can be rewritten as 
\begin{equation}
x^{2}+ \frac{b \eta}{a \eta-1}+ \frac{\eta}{a \eta-1}=0. \label{BT7}
\end{equation}
The above equation has two real roots with different sign, say $x_1,x_2$, moreover the positive root is given by
\begin{equation}
x_2= \frac{1}{2(a \eta-1)} (-b \eta - \sqrt{b^{2} \eta^{2}-4(a \eta-1) \eta} ). \label{BT10}
\end{equation}
From \eqref{BT3} the value for $y$ is:
\begin{equation}
y_2= \dfrac{(ax_2^{2} + bx_2+1)^{2}}{x_2(bx_2+2)} \left(1-2x_2 - \dfrac{hc}{(c+x_2)^{2}}- \delta  \right).
\end{equation}
Substituting $x_2$ and $y_2$ in equilibria equations we arrive to the following equations:
\begin{align*}
\delta &= a_1(x_2)+a_2(x_2)h, \\
h &= b_1(x_2)+b_2(x_2) \delta,
\end{align*}
where
\begin{align}
a_1(x_2) &= \dfrac{\eta (a x_2^{2}+b x_2 +1)^{2} \left( 1-2x_2 \right)}{x_2(bx_2+2)(m+x_2)+ \eta (a x_2^{2}+b x_2+1)^{2}}, \nonumber  \\
a_2(x_2) &=  \dfrac{\eta (a x_2^{2}+b x_2 +1)^{2}}{x_2(bx_2+2)(m+x_2)+ \eta (a x_2^{2}+b x_2+1)^{2}} \left( - \frac{c}{(c+x_2)^{2}} \right), \label{BT11} \\
b_1(x_2)&=\frac{(c+x_2)^{2} \left( (1-x_2)(bx_2+2)-(ax_2^{2}+bx_2+1)(1- 2x_2) \right) }{(bx_2+2)(c+x_2)-c(ax_2^{2}+bx_2+1)}, \nonumber \\
b_2 (x_2)&=  \frac{(c+x_2)^{2}(ax_2^2+bx_2+1)}{(bx_2+2)(c+x_2)-c(ax_2^{2}+bx_2+1)} \nonumber.
\end{align}

Solving previous equations we have 
\begin{equation}
h = \frac{b_1(x_2)+b_2(x_2) a_1(x_2) }{1-b_2(x_2) a_2(x_2)}, \quad \delta =a_1(x_2)+a_2(x_2) \left( \frac{b_1(x_2)+b_2(x_2) a_1(x_2) }{1-b_2(x_2) a_2(x_2)} \right).
\end{equation}
 
As in previous case, if $y_2$ satisfies the equilibria equation \eqref{ec6}, and then it is positive when $x_2$ is positive. 
\begin{theorem}
Let $c, \eta, m>0$, $0<a< \frac{1}{\eta}$, $b>-2 \sqrt{a}$ and $x_2$ as \eqref{BT10}. Set $a_1,a_2,b_1,b_2$ as \eqref{BT11}. If
\begin{equation}
h = \frac{b_1+b_2 a_1 }{1-b_2 a_2}, \quad \delta =a_1+a_2 \left( \frac{b_1+b_2 a_1 }{1-b_2 a_2} \right),
\end{equation}
whenever they are positive, then the system has an equilibrium at $(x_2,y_2)$ with a double zero eigenvalue, with $x_2>0$ and $y_2>0$. \label{BTteo2}
\end{theorem}

\subsubsection{Case $a \eta >1$} 
This case is similar to $a \eta<1$. Equation \eqref{BT2} can be rewritten as \eqref{BT7}, which has two roots: real with same sign or complex conjugate (depending on the discriminant), given by

\begin{equation}
x_{3,4}= \frac{1}{2(a \eta-1)} (-b \eta \pm \sqrt{b^{2} \eta^{2}-4(a \eta-1) \eta} ). \label{BT12}
\end{equation}

To avoid complex values for $x$, we ask $ b^{2} \eta^{2}-4(a \eta-1) \eta \geq 0 $ or equivalently $b^{2} \eta-4 (a \eta-1)\geq 0$, under this assumption, $x_3$ and $x_4$ are both real with same sign. Moreover, using expression \eqref{BT7} both are positive iff $b<0$ and $x_1=x_2= \frac{-b \eta}{2(a \eta-1)}$ when $b^{2} \eta-4 (a \eta-1)= 0$.

Again, from \eqref{BT3} the value of $y$ for each $x_i$ is:
\begin{equation}
y_i= \dfrac{(ax_i^{2} + bx_i+1)^{2}}{x_i(bx_i+2)} \left(1-2x_i - \dfrac{hc}{(c+x_i)^{2}}- \delta  \right).
\end{equation}
And substituting $x_i$ and $y_i$ ($i=3,4$) in equilibria equations
\begin{align}
h &= \frac{b_1(x_i)+b_2(x_i) a_1(x_i) }{1-b_2(x_i) a_2(x_i)}=:h_i\\
 \delta &=a_1(x_i)+a_2(x_i) \left( \frac{b_1(x_i)+b_2(x_i) a_1(x_i) }{1-b_2(x_i) a_2(x_i)} \right)=: \delta_i, i=3,4.
\end{align}
 As in previous case, whenever $(x_i,y_i)$ satisfy the equilibria equations, then $y_i>0$. 
\begin{theorem}
Let $c, \eta, m>0$, $a> \frac{1}{\eta}$ and $0>b>-2 \sqrt{a}$.
\begin{enumerate}
 \item If $ b^{2} \eta-4 (a \eta-1)<0 $, there is no  equilibrium points with double zero eigenvalues.
 \item If $b^{2} \eta-4 (a \eta-1)\geq 0$, let $x_3$ ($x_4$) as \eqref{BT12}. If $h=h_3$ ($h_4$) and $\delta= \delta_3$ ($\delta_4$) then the system has an equilibrium: $(x_3,y_3)$ (or ($x_4, y_4$) ),  with a double zero eigenvalue, with $x_i>0$ and $y_i>0$. Moreover, $(x_3,y_3)=(x_4,y_4)$ when $b^{2} \eta-4 (a \eta-1)=0$.
\end{enumerate} \label{BTteo3}
\end{theorem}
			
		\subsubsection{Normal form}

Take $h, \delta$ as bifurcation parameters and let $E=(x_1,y_1)$ a positive equilibrium of system \eqref{ec2}, that presents a double zero value at the Bogdanov-Takens point $h=h_{BT}, \delta= \delta_{BT}$. The Jacobian matrix of system at $(x_1,y_1)$ for an arbitrary value of $h$ and $\delta$ is given by \eqref{ec10}:
\begin{equation*}
J(E_1) = \left( 
\begin{matrix}
1-2x_1- \dfrac{x_1y_1(bx_1+2)}{(a(x_1)^{2} + bx_1+1)^{2}} - \dfrac{hc}{(c+x_1)^{2}} & -\dfrac{(x_1)^{2}}{a(x_1)^{2}+bx_1+1} \\
\dfrac{\delta^{2}}{\eta} & - \delta 
\end{matrix}
\right). 
\end{equation*}
 We transform system with the change $u=x-x_1, v=y-y_1$, obtaining:
\begin{align*}
u'&= \left( u+{ x_1} \right)  \left( 1-u-{ x_1} \right) -{\frac {
 \left( u+{ x_1} \right) ^{2} \left( v+{ y_1} \right) }{a \left( u
+{x_1} \right) ^{2}+b \left( u+{ x_1} \right) +1}}-{\frac {h
 \left( u+{ x_1} \right) }{c+u+{ x_1}}}, \\
v' &= \left( v+{ y_1} \right)  \left( \delta-{\frac {\eta\, \left( v+{
y_1} \right) }{m+u+{ x_1}}} \right) .
\end{align*}
In order to move the bifurcation parameters at $(0,0)$ (similar to equilibrium), let $\lambda=(\lambda_1, \lambda_2)$ and consider a perturbation of system in form $h=h_{BT}+ \lambda_1, \delta=\delta_{BT}+ \lambda_2$. Then previous system is rewritten as:
\begin{align}
u'&=\left( u+{ x_1} \right)  \left( 1-u-{ x_1} \right) -{\frac {
 \left( u+{x_1} \right) ^{2} \left( v+{y_1} \right) }{a \left( u
+{x_1} \right) ^{2}+b \left( u+{x_1} \right) +1}}-{\frac {
 \left( h_{BT}+{\lambda_1} \right)  \left( u+{x_1} \right) }{c
+u+{x_1}}}\\
& =: g_1((u,v), \lambda), \nonumber \\
v'&= \left( v+{ y_1} \right)  \left(  \delta_{BT}+\lambda_2-{\frac 
{\eta\, \left( v+{ y_1} \right) }{m+u+{x_1}}} \right)=: g_2((u,v), \lambda) . \label{ec11}
\end{align}

Or in short form $(u',v')^T= g((u,v), \lambda)$ with $g=(g_1,g_2)$. Note that the Jacobian matrix of system \eqref{ec11} at $(u,v)=(0,0)$ and $\lambda=0$ is equivalent to \eqref{ec10} evaluated at $(x_1,y_1)$ and $(h_{BT}, \delta_{BT})$. Therefore, if we denote the Jacobian matrix of \eqref{ec11} as $J((u,v), \lambda)$ then at $\lambda=0$, $(u,v)=0$, $J((0,0),0)$ has a double zero eigenvalue and it is equivalent to:

\begin{align}
1-2x_1- \dfrac{x_1y_1(bx_1+2)}{(a(x_1)^{2} + bx_1+1)^{2}} - \dfrac{h^{BT}c}{(c+x_1)^{2}} &= \delta^{BT}, \nonumber \\
\dfrac{(x_1)^{2}}{a(x_1)^{2}+bx_1+1}= \eta. \label{ec13}
\end{align}
So,
\begin{equation}
J((0,0),0)=\left( \begin{matrix}
\delta_{BT} & - \eta \\ \frac{{\delta_{BT}}^2}{\eta} & - \delta_{BT} 
\end{matrix} \right) =: J_0
\end{equation}

Let $v_0,v_1$ the generalized eigenvectors of $J_0$ and $w_0,w_1$ the generalized eigenvectors of $J_0^{T}$, given by:
\begin{equation*}
v_0 = \left( \begin {array}{c} \eta\\ \noalign{\medskip}\delta_{BT}\end {array}
 \right), \quad 
v_1 = \left( \begin {array}{c} \eta\\ \noalign{\medskip}\delta_{BT}-1\end {array}
 \right), \quad 
w_0 = \left( \begin {array}{c} - \frac{\delta_{BT}-1}{\eta} \\ \noalign{\medskip} 1 \end {array} \right), \quad
w1= \left( \begin {array}{c} \frac{\delta_{BT}}{\eta} \\ \noalign{\medskip}-1\end {array}
\right),
\end{equation*}

which satisfies $J_0 v_0=0$, $J_0 v_1=v_0$, $J_0^{T}w_1=0$, $J_0^{T}w_0=w_1$ and $ \langle v_1,w_1 \rangle = \langle v_0,w_0 \rangle =1  $, $\langle v_0,w_1 \rangle = \langle v_1,w_0 \rangle = 0 $. For matrix $P=[v_0 | v_1]$, define the change of variable $$(Y_1,Y_2)^T= P^{-1} (u,v)^T,$$

where $P$ has the property 
$$ P^{-1} J_0 P = \left( \begin{matrix}
0 &1 \\0 & 0
\end{matrix}
\right). $$

Then system \eqref{ec11} can be rewritten as 
\begin{align*}
Y_1 ' &= \langle g( Y_1 v_0+Y_2v_1, \lambda ),w_0 \rangle,  \\
Y_2' &= \langle g( Y_1 v_0+Y_2v_1, \lambda ),w_1 \rangle.
\end{align*}

Expanding the products above with Taylor expansion, we obtain the system 
\begin{align}
Y_1' &= Y_2+a_{00}(\lambda)  + a_{10}(\lambda) Y_1 +a_{01}(\lambda) Y_2+ \frac{1}{2} a_{20}(\lambda) Y_1^2+ a_{11}(\lambda)Y_1 Y_2 \nonumber \\
& + \frac{1}{2} a_{02}(\lambda)Y_2^2+..., \label{ec14} \\
 Y_2' &= b_{00}(\lambda)  + b_{10}(\lambda) Y_1 +b_{01}(\lambda) Y_2+ \frac{1}{2} b_{20}(\lambda) Y_1^2+ b_{11}(\lambda)Y_1 Y_2\nonumber\\
 & + \frac{1}{2} b_{02}(\lambda)Y_2^2+..., \nonumber
\end{align}

With help of Maple, we compute each coefficient $a_{ij}, b_{ij}$, and then we use the equilibria equations \eqref{ec12} and the trace and determinant equations \eqref{ec13} to simplify them. We obtain:

\begin{align*}
a_{00}(\lambda)&= {\frac {{ \lambda_1}\,{x_1}\, \left( { \delta^{BT}}-1 \right) }{
 \left( c+{ x_1} \right) \eta}}+{y_1}\,{ \lambda_2}, \\
 a_{10}(\lambda) &= \frac{(\delta^{BT}-1) c \lambda_1}{(c+x_1)^2 } + \delta^{BT} \lambda_2, \\
 a_{01}(\lambda) &= \frac{(\delta^{BT}-1) c \lambda_1}{(c+x_1)^2 } + (\delta^{BT}-1) \lambda_2,
\end{align*}
\begin{align*} 
 a_{20}(\lambda) &=  - \frac{2(\delta^{BT}-1)}{\eta} \left[  - \frac{1}{ax_1^2+bx_1+1} \right[ -\frac{x_1^2y_1a \eta^2}{ax_1^2+bx_1+1}+\\
 &+ \eta^2y_1+2x_1 \eta \delta^{BT}- \frac{x_1(2a \eta x_1+b \eta)}{(ax_1^2+b x_1+1)^2} \left( x_1 \delta^{BT}(ax_1^2+bx_1+1)+b \eta x_1y_1+2 \eta y_1 \right)   \left] \right] + \\
& - 2(\delta^{BT}-1)\left( - \eta + \frac{(h^{BT}+ \lambda_1)\eta c}{(c+x_1)^3} \right), \\
a_{11}(\lambda) &= - \frac{2}{\eta} (\delta^{BT}-1) \left( - \eta^{2}- \frac{\eta(\eta y_1 + 2x_1 \delta^{BT}-x_1)}{(ax_1^2+bx_1+1)}+ \frac{x_1^{2}(2 a \eta x_1+b \eta )(2 \delta^{BT}-1)}{2(ax_1^2+bx_1+1)^2} \right)+ \\
&- \frac{2}{\eta} (\delta^{BT}-1) \left( \frac{x_1(b \eta x_1y_1 + 2 \eta y_1)(2a \eta x_1+b \eta)}{(ax_1^2+bx_1+1)^3} + \frac{x_1^{2}y_1 a \eta^{2}}{ax_1^{2}+bx_1+1} + \frac{(h^{BT}+ \lambda_1) \eta^{2}c}{(c+x_1)^3} \right), \\
a_{02}(\lambda) &= - \frac{2(\delta^{BT}-1)}{\eta} \left[ - \eta^{2} + \frac{(h^{BT}+ \lambda_1)\eta^{2}c}{(c+x_1)^3} - \frac{1}{ax_1^{2}+bx_1+1} \right[2x_1 \eta (\delta^{BT}-1)+ \\
&- x_1 \left( \frac{\delta^{BT} x_1}{ax_1^{2}+bx_1+1} + \frac{b \eta x_1 y_1+2 \eta y_1}{(ax_1^{2}+bx_1+1)^2} - \frac{x_1}{ax_1^2+bx_1+1} \right) (2 a \eta x_1+b \eta) + \\
& - \frac{x_1^2y_1 a \eta^{2}}{a x_1^2+bx_1+1}+ \eta^{2}y_1 \left] \right]+ \frac{2\eta(\delta^{BT}-1)}{m+x1}- \frac{2\eta^2y_1}{(m+x_1^2)},
\end{align*}

\begin{align*}
b_{00} (\lambda)&= \frac{- \delta^{BT} x_1 \lambda_1 }{\eta( c+x_1 )}- \lambda_2 y_1, \\
b_{10} (\lambda)&=- \frac{c \delta^{BT}  \lambda_1 }{(c+x_1)^2}- \delta^{BT} \lambda_2, \\
b_{01} (\lambda)&= \frac{- c \delta^{BT} \lambda_1}{(c+x_1)^2}- ( \delta^{BT}-1 ) \lambda_2, \\
b_{20} (\lambda)&= \frac{2 \delta^{BT}}{\eta} \left[ - \frac{1}{ax_1^2+bx_1+1} \right( - \frac{x_1^2y_1a \eta^{2} }{ax_1^2+bx_1+1} + \eta^{2}y_1 + 2 x_1 \eta \delta^{BT}+ \\
&- \frac{x_1^{2} \delta^{BT}(2 a \eta x_1+b \eta)}{(ax_1^{2}+bx_1+1)}- \frac{(2 a \eta x_1 + b \eta)x_1 \eta y_1 (bx_1+2)}{(ax_1^2+bx_1+1)^{2}} \left) - \eta^{2} + \frac{(h^{BT}+ \lambda_1)\eta^{2}c}{(c+x_1)^3} \right],\\
\end{align*}
\begin{align*}
b_{11} (\lambda)&= \frac{\delta^{BT}}{\eta} \left[-2 \eta^{2} - \frac{1}{ax_1^{2}+bx_1+1} ( 2 \eta^{2} y_1 + 4x_1 \eta \delta^{BT} - 2x_1 \eta) \right] + \\
& \frac{ \delta^{BT} x_1(2 a \eta x_1+ b \eta)}{ \eta(ax_1^2+bx_1+1)^{3}} [ x_1 (ax_1^{2}+bx_1+1)(2 \delta^{BT}-1 )+2b \eta x_1 y_1+4 \eta y_1 ]+ \\
& + \frac{2a \eta \delta^{BT} x_1^{2}y_1}{ (ax_1^{2}+bx_1+1)^{2}} + \frac{2 \delta^{BT} (h^{BT}+ \lambda_1)\eta c }{(c+x_1)^{3}}, \\
b_{02}(\lambda) &= -2 \delta^{BT} \left( - \eta + \frac{(h^{BT}+ \lambda_1)\eta c}{(c+x_1)^{3}}\right) - \frac{ 2\delta^{BT}}{\eta ( ax_1^{2}+bx_1+1 )} \left[ - \frac{x_1(2a \eta x_1+b \eta)}{(ax_1^2+bx_1+1)^2} \right( \\
&x_1(ax_1^2+bx_1+1)(\delta^{BT}-1)+ b \eta x_1y_1+ 2 \eta y_1 \left) - \frac{x_1^{2}y_1 a \eta^{2}}{ax_1^{2}+bx_1+1}+ 2x_1 \eta (\delta^{BT}-1)+ \eta^{2}y_1 \right]\\
&+ \frac{2 \eta}{m+x_1} .
\end{align*}
Set $u_1=Y_1$ and $u_2 $ the right hand of first the first equation in \eqref{ec14}, then system \eqref{ec14} is transformed into
\begin{align*}
u_1'&=u_2, \\
u_2'&= g_{00}(\lambda)+ g_{10}(\lambda)u_1+ g_{01}(\lambda)u_2+ \frac{1}{2} g_{20}(\lambda)u_1^{2} + g_{11}(\lambda)u_1 u_2 + \frac{1}{2}g_{02}(\lambda)u_2^{2}+ Q(u_1,u_2, \lambda),
\end{align*}
where $Q(u_1,u_2, \lambda)=O(\Vert u \Vert^{3})$ and the relevant terms of $g_{ij}$ are given by $g_{00}(0)=g_{10}(0)=g_{01}(0)=0$,
\begin{align*}
g_{20}(0)&=b_{20}(0), \\
g_{11}(0)&=a_{20}(0)+b_{11}(0),\\ 
 g_{02}(0)&=b_{02}(0)+2a_{11}(0), \\
g_{00}(\lambda) &=  b_{00}(\lambda)+..., \\
g_{10}(\lambda) &= b_{10}(\lambda)+a_{11}(\lambda)b_{00}(\lambda)-b_{11}(\lambda)a_{00}(\lambda)+..., \\
g_{01}(\lambda) &= b_{01} (\lambda)+a_{10}(\lambda)+a_{02}(\lambda) b_{00}(\lambda)-(a_{11}(\lambda) + b_{02}(\lambda))a_{00}(\lambda)+... ,
\end{align*}
where the displayed terms are sufficient to compute the first partial derivatives of $g_{00}(\lambda),g_{10}(\lambda),g_{01}(\lambda)$.
Assume that $g_{11}(0)= a_{20}(0)+b_{11}(0) \neq 0$ (BT.1),
then we can make a parameter shift of coordinates in the $u_1$- direction with $u_1=v_1+ \delta(\lambda), u_2=v_2$ and $ \delta(\lambda)\approx - \frac{g_{01}(\lambda)}{g_{11}(0)} $, then 

\begin{align*}
v_1' &= v_2, \\
v_2' &= h_{00}(\lambda )+h_{10}(\lambda )v_1+ \frac{1}{2}h_{20}(\lambda )v_1^{2}+h_{11}(\lambda )v_1 v_2 + \frac{1}{2}h_{02}(\lambda )v_2^{2}+...
\end{align*}
where  $$h_{20}(0)= g_{20}(0), \quad h_{11}(0)=g_{11}(0), \quad h_{02}(0)=g_{02}(0),$$
and the relevant terms of $h_{kl}$ to compute the first partial derivatives are given by 
\begin{align*}
h_{00}(\lambda) &=g_{00}(\lambda)+... \\
h_{10} (\lambda)&= g_{10}(\lambda)- \frac{g_{20}(0)}{g_{11}(0)} g_{01}(\lambda)+...
\end{align*} 
Introducing a new time via the equation $ dt=(1+ \theta v_1)d \tau $ and $\theta (\lambda)=- \frac{h_{02}(\lambda)}{2}$ we have
\begin{align}
\zeta'_1 &= \zeta_2, \nonumber \\
\zeta_2' &= \mu_1(\lambda)+ \mu_2(\lambda) \zeta_1 + A(\lambda) \zeta_1^{2} + B(\lambda) \zeta_1 \zeta_2 +..., \label{ec15}
\end{align}
where 
\begin{equation}
\mu_1(\lambda)=h_{00}(\lambda), \quad \mu_2(\lambda)=h_{10}(\lambda)- \frac{1}{2}h_{00}(\lambda)h_{02}( \lambda ),
\end{equation}
and
\begin{equation}
A(\lambda)= \frac{1}{2}(h_{20}(\lambda)-h_{10}(\lambda)h_{02}(\lambda)), \quad B(\lambda)=h_{11}(\lambda).
\end{equation}
If we assume $2A(0)= b_{20}(0) \neq 0$ (BT.2),
then we can introduce a new time scaling (denoted by $t$ again) and new variables $\eta_1, \eta_2$, given by
\begin{equation}
t= \left| \frac{B(\lambda)}{A(\lambda)} \right|, \quad \eta_1= \frac{A(\lambda)}{B^{2}(\lambda)} \zeta_1, \quad \eta_2 = sign \left( \frac{B(\lambda)}{A(\lambda)} \right)\frac{A^{2}(\lambda)}{B^{3}(\lambda)} \zeta_2,
\end{equation}
in the coordinates $(\eta_1, \eta_2)$, the system \eqref{ec15} takes the form
\begin{align}
\eta_1'&= \eta_2, \nonumber \\
\eta_2'&= \beta_1+ \beta_2 \eta_1 + \eta_1^{2}+ s \eta_1 \eta_2 + O(\Vert \eta \Vert^{3}) \label{ec16} ,
\end{align}
with $s= sgn( b_{20}(a_{20}(0)+b_{11}(0)) )$
\begin{align*}
\beta_1(\lambda) &= \frac{B^{4}(\lambda)}{A^{3}(\lambda)} \mu_1(\lambda), \\
\beta_2(\lambda) &= \frac{B^{2}(\lambda)}{A^{2}(\lambda)} \mu_2(\lambda).
\end{align*}
In order to define an invertible smooth change of parameters near $\lambda=0$, we also assume $$ \det (\frac{\partial \beta}{\partial \lambda }(\lambda=0)) \neq 0 \quad \text{(BT.3)} .$$ Using 8.4 from \cite{kuznetsov2013elements}, we summarize the previous analysis in the following theorem

\begin{theorem}
Let $E^{*}=(x_1,y_1)$ an equilibrium point with a double zero eigenvalue . If $\delta, h$ are chosen as bifurcation parameters,  $ a_{20}(0)+b_{11}(0) \neq 0 $, $b_{20}(0) \neq 0$,  are satisfied and the  matrix $\left( \frac{\partial (\beta_1, \beta_2)}{\partial (\lambda_1, \lambda_2)}  \right)\mid_{\lambda=0} $ is non-singular,
then there exists smooth invertible variable transformations smoothly depending on parameters, a direction preserving time reparametrization and smooth invertible parameter changes, which reduces the  system \eqref{ec2} to \eqref{ec16}, so, system \eqref{ec2} undergoes a Bogdanov-Takens bifurcation in a small neighbourhood of $E^{*}$ as $h, \delta$ vary near $h^{BT}, \delta^{BT}$.
\end{theorem}

%\begin{theorem}
%Let $E^{*}=(x_1,y_1)$ an equilibrium point with a double zero eigenvalue . If $\delta, h$ are chosen as bifurcation parameters,  $ a_{20}(0)+b_{11}(0) \neq 0 $, $b_{20}(0) \neq 0$,  are satisfied and the  matrix $\left( \frac{\partial (\beta_1, \beta_2)}{\partial (\lambda_1, \lambda_2)}  \right)\mid_{\lambda=0} $ is nonsingular,
%then  system \eqref{ec2} undergoes a Bogdanov-Takens bifurcation in a small neighborhood of $E^{*}$ as $h, \delta$ vary near $h^{BT}, \delta^{BT}$.
%\end{theorem}
From \cite{kuznetsov2013elements} (chapter 8 section 4) we know that the bifurcation curves can be approximated for small values of $(\lambda_1, \lambda_2)$ by:
\begin{align}
T&= \{ (\lambda_1, \lambda_2), 4 \beta_1(\lambda_1, \lambda_2)- (\beta_2(\lambda_1, \lambda_2))^2=0 \}, \nonumber \\
H&= \{ (\lambda_1, \lambda_2),  \beta_1(\lambda_1, \lambda_2) =0, \beta_2<0 \}, \label{ec21} \\
P &= \{ (\lambda_1, \lambda_2), \beta_1(\lambda_1, \lambda_2)=- \frac{6}{25}  (\beta_2(\lambda_1, \lambda_2))^2=0, \beta_2<0 \}. \nonumber
\end{align}
The curve $T$ divides the plane $\lambda_1- \lambda_2$ in two zones, one of them with two equilibria points and the other with no equilibria. On this curve there exists only one equilibrium. The curve $H$ corresponds to the existence of Hopf bifurcation and the existence of a limit cycle (stable if $s<0$ and unstable if $s>0$) curve $P$ for the existence of a homoclinic loop. 
			
		\subsubsection{Numerical simulations}

\begin{example}
Let the parameters be defined as follows: $\eta=0.1, a=2,b=-2.82,c=0.05,m=0.8$, then we have  $b+ 2 \sqrt{a}=0.008427124>0$ which implies $p(x)>0$ as required in introduction. Computing, $a \eta=0.2<1 $, so using   theorem \eqref{BTteo2} we have an equilibrium with double zero eigenvalue at $h^{BT}=0.1715598183, \delta^{BT} =0.03070149222 $. \par 

For these values, system \eqref{ec3} is rewritten as 
\begin{align}
\frac{dx}{dt} &= x \left( 1-x \right) -{\frac {{x}^{2}y}{2\,{x}^{2}- 2.82\,x+1}}-
 0.1715598183\,{\frac {x}{ 0.05+x}}, \label{ec22} \\
\frac{dy}{dt} &= y \left(  0.03070149222- 0.1\,{\frac {y}{ 0.8+x}} \right) . \nonumber
\end{align}

The trivial equilibria points of system are: 
\begin{align*}
E&=(0,0), E_y =( 0, 0.2456119378), \\
E^{+}&=(0.7975913540, 0), E^{-}=(0.1524086460,0 ), 
\end{align*} 
 
 and there is a single interior equilibrium  with zero double eigenvalue  given by $$ E_1=(0.2187994431,0.3127866314).$$
With help of Maple software, we compute the change of variable derived in previous sections, in order to compute the Bogdanov Takens conditions BT1, BT2, BT3, obtaining:
 \begin{align*}
 &g_{11}(0)=-0.5922764628 \neq 0, 2A(0,0)= -0.01942828012 \neq 0, \\
 &  \det \left( \frac{\partial (\beta_1, \beta_2)}{\partial (\lambda_1, \lambda_2)}  \right)\mid_{\lambda=0}= 3.559954288^* 10^6 .
 \end{align*}
  Computing the coefficient $s$ of the normal form in \eqref{ec16}, we arrive to $s=0.01150691302>0$, so the limit cycle will be unstable. \par 

\begin{figure}
\subfigure[] {\includegraphics[scale=0.5]{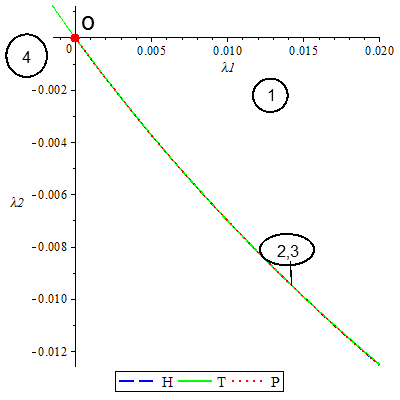}}
\subfigure[] {\includegraphics[scale=0.8]{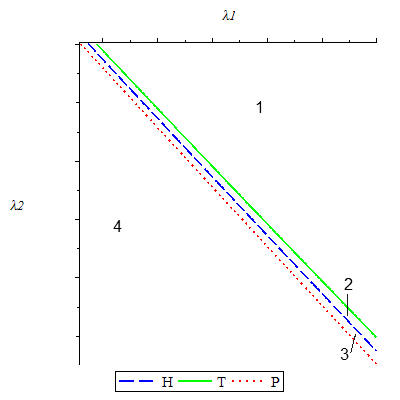}}
\caption{Local approximation of bifurcation curves for values $(\lambda_1, \lambda_2)$ near $(0,0)$ for $\eta=0.1, a=2,b=-2.82,c=0.05,m=0.8$.  For instance, in a) the three curves are very closely so they look like a single curve, but in b) the zones between each one can be observed, with a zoom. } \label{fig6}
\end{figure}

From \eqref{ec21} we have the local representation of bifurcations curves, which are plotted in figure \eqref{fig6} with the sections between them, and define the behaviour of interior equilibria points for values of $(\lambda_1, \lambda_2)$ near $(0,0)$.  \par 
For $h^{BT}=0.1715598183, \delta^{BT} =0.03070149222 $, we take $ \lambda_1=0.02$  and vary $\lambda_2$ at each section of figure \eqref{fig6}. Setting $h= h^{BT}+ \lambda_1, \delta= \delta^{BT}+ \lambda_2$ to plot the phase portrait of system \eqref{ec22} in a neighbourhood of $E_1=(0.2187994431,0.3127866314)$ and using the theoretical results given in \cite{kuznetsov2013elements} about the phase portrait on each zone, we obtain the following:
\begin{enumerate}

\begin{figure}
\subfigure[] {\includegraphics[scale=0.5]{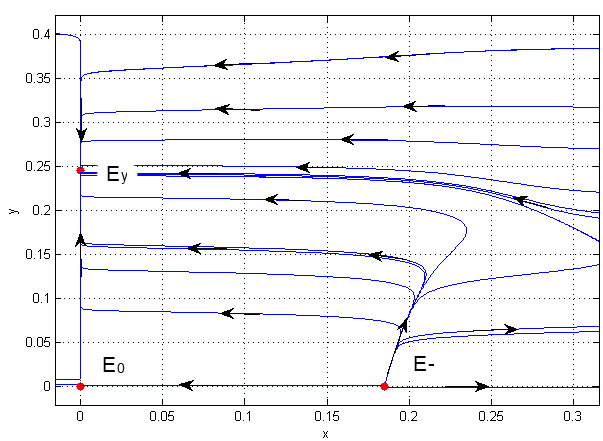}}
\subfigure[] {\includegraphics[scale=0.5]{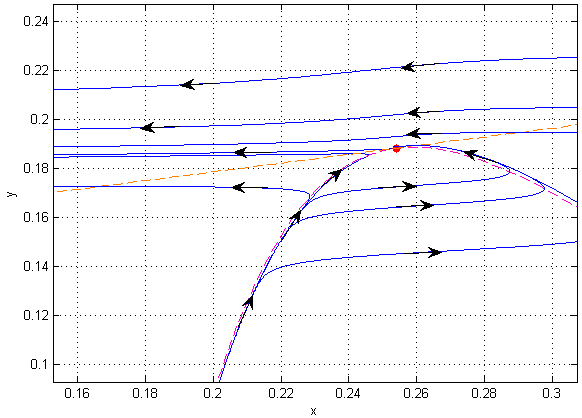}}
\caption{ Phase portrait of system \eqref{ec22} for $\lambda_1=0.02$. a) $\lambda_2=0$, $h=0.1915598183$ and $\delta=0.03070149222$. There is no interior equilibria points. b) $ \lambda_2=-0.01283735222$, $h=0.1915598183$ and $\delta=0.01786414000$. An interior equilibrium appears.} \label{fig7}
\end{figure}

\item In section 1, above the curve T there is no interior equilibrium points and around the curve $T$, a single equilibrium $E^{*}$ with a zero eigenvalue appears. Figure \eqref{fig7}.

\begin{figure}
\begin{center}
\includegraphics[scale=0.6]{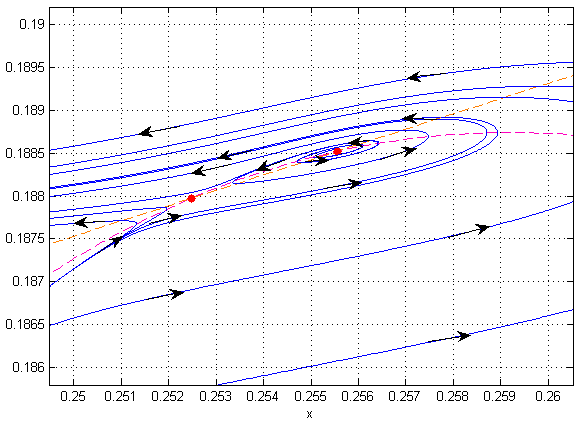} 
\end{center} \caption{Phase portrait of system \eqref{ec22} for $\lambda_1=0.02$ and $\lambda_2=-0.01284149222$. $h=0.1915598183 $ and $\delta=0.01786$, there are two equilibria: a saddle and a spiral source.} \label{fig8}
\end{figure}

\begin{figure}
\subfigure[]{\includegraphics[scale=0.5]{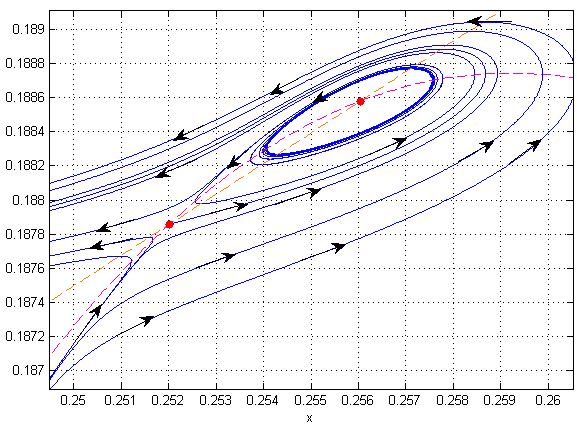}}
\subfigure[]{\includegraphics[scale=0.5]{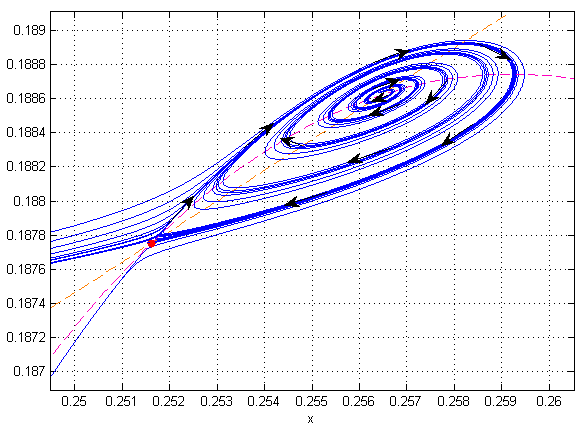}}
\caption{Phase portrait of system \eqref{ec22} for $\lambda_1=0.02$ . a) $\lambda_2=-0.01284449222$, $h=0.1915598183 $ and $ \delta=0.017857 $, there exists an unstable limit cycle. b) Approximation of homoclinic loop with $\lambda_2=-0.01284767222$, $h=0.1915598183, \delta=0.01785382 $.} \label{fig9}
\end{figure}

\item In section 2, between $T$ and $H$ the equilibrium $E^{*}$ is divided in two equilibria : $E_1$ which is a saddle and $E_2$ which is a spiral source (unstable). $E_1$ and $E_2$ are closer as $(\lambda_1, \lambda_2)$ remains close to $T$. Figure \eqref{fig8}.
\item The curve $H$ corresponds to a Hopf Bifurcation  of equilibrium $E_2$, which changes its stability from source (unstable) to a nodal sink (stable). Section 3 presents the existence of an unstable limit cycle around $E_2$. $E_2$ becomes into a nodal sink and $E_1$ remains as a saddle. The orbit of limit cycle becomes closer and closer to $E_1$ as $(\lambda_1, \lambda_2)$ goes to curve $P$. On the curve $P$, the limit cycle around $E_2$ becomes into a homoclinic loop. Figure \eqref{fig9}.
\item In section 4 the homoclinic loop disappears and $E_2$ remains stable while $E_1$ is always a saddle.
\item At point $O$ where all curves intersect, we have the Bogdanov-Takens point, where there exist a single equilibrium $E_1=(0.2187994431,0.3127866314)$ with double zero eigenvalue.  
\end{enumerate}

\end{example}

\section{Conclusions}			

The predator-prey models have been extensively studied by mathematical and biological researchers since its introduction made by Lotka and Volterra. Its importance lies in understanding the dynamics between two species (a predator and a prey) that live together in the same environment, in order to look for suitable conditions that allow the both species survive in equilibria. However, several authors (see for example \cite{gupta2013bifurcation}, \cite{gupta2015dynamical}, \cite{hu2017stability}) have shown that considering a harvesting term in the model can lead to the extinction of any species. \par  
In this paper we describe the dynamics and bifurcations of a predator-prey system with  functional response of Holling type III, that considers a Michaelis-Menten harvesting term in prey population.  The choice of the functional response as Holling type III and the harvesting term, gave rise to a wide variety of scenarios for the existence of positive (and total) equilibria points. The equilibria points obtained were of two kinds: trivial, with a component equal zero, that represents extinction of any population; and interior, where its components are not zero and both species exist. The interior points were located in three zones of existence: $K_1, K_2, K_3$, depending on the sign of $h-c$, this fact suggests that the number of non-trivial equilibria points admitted for the system depends strongly by the harvesting rate (the parameters $h$ and $c$ are the result of a re-scaling in the original harvesting term). All the possible cases  were  mathematically described. \par  
 Four trivial equilibria points were obtained: the extinction point $E=(0,0)$ (where there is no predator neither prey), two predator-free: $E^+, E^-$ (where there is only prey population) and the prey extinction $E_y$. We determined the stability of trivial points via the linearization of the system around each one. The results say that the extinction point $E$ could be a saddle or an unstable node, for $h-c \neq 0$. When $h-c=0$, the equilibrium presents a saddle-node bifurcation, and at $h=c$ the equilibrium is a saddle or a saddle node. In all cases, the extinction equilibrium is unstable, so there is no possibility (under the assumptions of our model) that both species go to extinct at same time. This phenomena appears also for the predator free equilibria, where both are always unstable, indicating that the population will never go to a state with preys and no predators. However, the extinction could be of preys at $E_y$. $E_y$ is locally asymptotically stable when $c<h$ (where both parameters, $h$ and $c$ are directly related to the harvesting rate of preys, and the carrying capacity of environment), and in this case we can have that predators survive by eating its alternative food and the preys go to extinct. This is a scenario that biologist try to avoid. \par
 Due to the variety of cases for the existence of interior equilibria, we do not compute the linearization  of system at all the interior equilibria points, instead of, we provide an extensive bifurcation analysis. When we fix all parameters and vary $\delta$, the system has an equilibrium $E^*$ which presents a Hopf bifurcation at $\delta= \delta^H$, making possible the existence of a limit cycle around $E^*$. The first Lyapunov coefficient was also calculated to determine the stability of the limit cycle. When $\delta$ and $h$ are taken as bifurcation parameters, the system presents an equilibrium with zero double eigenvalue at $\delta=\delta^{BT}, h=h^{BT}$, and therefore, a Bogdanov-Takens bifurcation of codimension two. The dynamics of the system for values of parameters near to the Bogdanov-Takens point are extensively described, obtaining the apparition of limit cycles and homoclinic loops.  The bifurcation parameters were taken as $\delta$ and $h$ following the references, but it will be interesting to make a bifurcation analysis varying only the harvesting parameters $h, c$.\par 
A Maple code was implemented to obtain numerically the approximation of the curves $T, H, P$ (for the existence of equilibria, Hopf bifurcation and homoclinic loop, respectively) that divide the plane of parameters  in the different phase portraits possibles in a Bogdanov-Takens bifurcation. Even when we obtain an approach of the curves that let us to find the limit cycle and the homoclinic loop, the order of approximation depends strongly in the  neighbourhood of $(\delta^{BT}, h^{BT})$ that is taken, so an smaller neighbourhood must give a better approach. 
In the biologically meaning, it is very interesting to try to validate the model that we propose in this article with real values. If this system results a good model for the real values, then we can take the mathematical results obtained in the harvesting parameter $h$, to define harvesting laws and restrictions that avoid the stability of the prey-extinction equilibria and allow the stability of an interior equilibrium, because in this scenario we will gain the coexistence of both species in a long time.

 \appendix
\section{Appendix: Method of Ferrari} \label{apen1}
	
Let the arbitrary equation 
\begin{equation}
P(x)=x^4+Ax^3+Bx^2+Cx+D=0.
\end{equation}
Introducing the change of variable (a Tchirnhausen substitution to eliminate the cubic term) $X=x+ A/4$, then the equation is equivalent to:
$$P(x)=P(X-A/4)=X^4+P_2X^3+Q_2X+r=Q(X),$$
where:
\begin{align*}
P_2 &= - \frac{3}{8}A^2+B,\\
Q_2 &= \frac{1}{8}A^3- \frac{1}{2}AB+C, \\
r&= - \frac{3}{256}A^4+ \frac{1}{16}A^2 B- \frac{1}{4} AC+D.
\end{align*}
Now, for an arbitrary $u$:
$$(X^2+ \frac{P_2}{2}+u)^2=X^4+P_2X^2+2X^2u+ \frac{1}{4} P_2^2 + P_2 u+ u^2,$$
so, we can rewrite $Q(X)$ as
\begin{align*}
Q(X) &= (X^2+ \frac{P_2}{2}+u )^2- \left[ 2X^2u-Q_2 X+ \left( u^2+P_2u+ \frac{1}{4}P_2^2-r \right) \right], \\
&=(X^2+ \frac{P_2}{2}+u )^2 -2u\left[ X^2- \frac{Q_2}{2u} X  + \left( \frac{u}{2}+ \frac{P_2}{2} + \frac{P_2^2}{8u}- \frac{r}{2u} \right) \right],
\end{align*}
whenever $u \neq 0$. To have a quadratic expression in brackets we ask for an $u$ such that
$$  \left( \frac{Q_2}{4u} \right)^2 = \frac{u}{2}+ \frac{P_2}{2} + \frac{P_2^2}{8u}- \frac{r}{2u}, $$
or equivalently
\begin{equation}
8u^3+8P_2u^2+2P_2^2u-8ru- Q_2^2=0. \label{ap1}
\end{equation}
Therefore, when $u$ satisfies equation \eqref{ap1}, $Q(X)$ has the following form:
\begin{align*}
Q(X) &= \left( X^2+ \frac{P_2}{2}+u \right)^2-2u \left( X- \frac{Q_2}{4u} \right)^2, \\
&= \left[ X^2+ \frac{P_2}{2} +u + \sqrt{2u} \left( X- \frac{Q_2}{4u} \right) \right] \left[ X^2+ \frac{P_2}{2} +u - \sqrt{2u} \left( X- \frac{Q_2}{4u} \right) \right], \\
&= q_1(X) q_2(X).
\end{align*}

We have transformed the quartic polynomial in two quadratic polynomials.
Note that $u$ is any solution of \eqref{ap1}, which is a cubic equation with
independent term $-Q_2^2 \leq 0$.  If $Q_2 \neq 0$, equation \eqref{ap1} has always a positive real root, say, $u+$. We will work with this positive root and omit the + sign for simplicity. Define 
\begin{align*}
\Delta_1 &= 2u - 4 \left( \frac{P_2}{2}+u- \frac{Q_2}{2 \sqrt{2u}} \right), \\
\Delta_2 &= 2u - 4 \left( \frac{P_2}{2}+u+ \frac{Q_2}{2 \sqrt{2u}} \right).
\end{align*}
The roots of $q_1(X), q_2(X)$ are given by:
$$ X_1^{\pm} = \frac{1}{2} \left( - \sqrt{2u} \pm \sqrt{\Delta_1} \right), \quad X_2^{\pm} = \frac{1}{2} \left(  \sqrt{2u} \pm \sqrt{\Delta_2} \right). $$ 
Therefore, the four roots of equation \eqref{ap1} are the following:
\begin{align}
x_1^{\pm} &= \frac{1}{2} \left( - \sqrt{2u} \pm \sqrt{\Delta_1}- \frac{A}{2} \right), \\
x_2^\pm &= \frac{1}{2} \left( \sqrt{2u} \pm \sqrt{\Delta_2}- \frac{A}{2} \right).
\end{align}

\bibliographystyle{plain}

\end{document}